\documentclass{article}
\usepackage[utf8]{inputenc}				%(only for the pdftex engine)
\usepackage{lmodern} 
\usepackage{microtype}
\usepackage[numbers,square,sort&compress]{natbib}

\usepackage{caption}
\usepackage{subcaption}
\usepackage{enumitem}
\usepackage{amsmath}
\usepackage{amssymb}
\usepackage{bbm}
\usepackage{verbatim}
\usepackage{graphicx,epstopdf}
\usepackage{array}
\usepackage{tikz-cd}
\usepackage{xfrac}
\usepackage{faktor}

\usepackage{braket,amsfonts}
\usepackage{multirow}
\usepackage{booktabs}
\usepackage{graphbox}
\usepackage{overpic}
\usepackage{algorithm}
\usepackage{algorithmic}

\usepackage{graphicx}
\usepackage{amsmath,amsfonts,amsthm}
\usepackage{amssymb,latexsym}
\usepackage{fixmath}
\usepackage{mathrsfs,amsbsy}
\usepackage{dsfont}
\usepackage{enumerate}
\usepackage{algorithm,algorithmic}
\usepackage{kbordermatrix}
\usepackage{xcolor}
\usepackage{caption}
\usepackage{subcaption}
\usepackage{enumitem}

\date{today}
\numberwithin{equation}{section}

\newtheorem{theorem}{Theorem}[section]
\newtheorem{proposition}{Proposition}[section]
\newtheorem{lemma}[theorem]{Lemma}

\newtheorem{remark}{Remark}%[section]
\newtheorem{example}{Example} 
\newcommand{\bw}{\textbf w}

\title{On Novel Fixed-Point-Type Iterations with Structure-Preserving Doubling Algorithms for Stochastic Continuous-time Algebraic Riccati equations}
\author{Tsung-Ming Huang\thanks{ Department of Mathematics,  National Taiwan Normal
    University, Taipei 116, Taiwan  ({\tt min@ntnu.edu.tw}).} \and
    Yueh-Cheng Kuo\thanks{Department
    of Applied Mathematics, National University of Kaohsiung,
    Kaohsiung 811, Taiwan ({\tt yckuo@nuk.edu.tw}).} \and
    Ren-Cang Li\thanks{ Department of Mathematics, 
    University of Texas at Arlington, Arlington, USA ({\tt rcli@uta.edu}).} \and
    Wen-Wei Lin\thanks{ Nanjing Center for Applied Mathematics,   
Nanjing, China; Department of Applied Mathematics, National Yang Ming Chiao Tung University, Hsinchu 300, Taiwan. ({\tt wwlin@math.nctu.edu.tw}).}}
\date{\today}
\begin{document}
\maketitle
\begin{abstract}
    In this paper we mainly propose efficient and reliable numerical algorithms for solving stochastic continuous-time algebraic Riccati equations (SCARE) typically arising from the differential state-dependent Riccati equation technique from the 3D missile/target engagement, the F16 aircraft flight control and the quadrotor optimal control etc. 
To this end, we develop a fixed point (FP)-type iteration with solving a CARE by the structure-preserving doubling algorithm (SDA) at each iterative step, called FP-CARE\_SDA. We prove that either the FP-CARE\_SDA is monotonically nondecreasing or nonincreasing, and is R-linearly convergent, with the zero initial matrix or a special initial matrix satisfying some assumptions. 
The FP-CARE\_SDA (FPC) algorithm can be regarded as a robust initial step to produce a good initial matrix, and then the modified Newton (mNT) method can be used by solving the corresponding Lyapunov equation with SDA (FPC-mNT-Lyap\_SDA). Numerical experiments show that the FPC-mNT-Lyap\_SDA algorithm outperforms the other existing algorithms.
\end{abstract}

\section{Introduction}
The nonlinear dynamics of the stochastic state-dependent control system in continuous-time subject to multiplicative white noises can be described as
\begin{align}
    d x(t) = A(x) x + B(x) u + \sum_{i=1}^r (A_0^i(x) x + B_0^i(x)u) d w_i(t), \label{eq:dyna_SSDC}
\end{align}
wher $A(x)$, $A_0^i(x) \in \mathbb{R}^{n \times n}$, $B(x)$, $B_0^i(x) \in \mathbb{R}^{n \times m}$ for $i = 1, \ldots, r$, $x(t)$ and $u(t)$ are the state and the control input, $w(t) = [ w_1(t), \cdots, w_r(t) ]^{\top}$ is a standard Wiener process satisfying that each $w_i(t)$ is a standard Brownian motion.
Under the nonlinear dynamical system \eqref{eq:dyna_SSDC}, we consider the cost functional with respect to the control $u(t)$ with a given initial $x_0$
\begin{align}
    J(t_0, x_0; u) = E \left\{ \int_{t_0}^{\infty} \begin{bmatrix} x \\ u \end{bmatrix}^{\top} \begin{bmatrix}
        Q(x) & L(x) \\ L^{\top}(x) & R(x)
    \end{bmatrix} \begin{bmatrix}
        x \\ u
    \end{bmatrix} dt\right\}, \label{eq:cost_fun}
\end{align}
in which $Q(x) \in \mathbb{R}^{n \times n}$, $L(x) \in \mathbb{R}^{m \times n}$, and $R(x) > 0 \in \mathbb{R}^{m \times m}$ with $Q(x) - L(x) R(x)^{-1} L(x)^{\top} \geq 0$.

%The recent research on target interception or trajectory tracking has become increasingly important, particularly in the design of missile guidance and unmanned aerial vehicles. 
%Amongst many strategies, the impact angle guidance (IAG) \cite{nakm:2021a} is an essential strategy that guides the missile/purseur to intercept the target at a specific angle and within a finite time. 
%Based on IAG, the 
Recently, the state-dependent Riccati equation (SDRE) \cite{base:2020} generalizes the well-known linear quadratic regulator (LQR) \cite{meccr:2019} and attacks broad attention in nonlinear optimal controls \cite{poli:2020, kikw:2017, neko:2019}.
The SDRE scheme manifests state and control weighting functions to ameliorate the overall performance \cite{Cime:2012}, as well as, capabilities and potentials of other performance merits as global asymptotic stability \cite{lixi:2019,lilc:2018,behe:2018}. 
In practical applications such as the differential SDRE with impact angle guidance strategies \cite{lwhxwl:2023,nakm:2021a} models a 3D pursuer/target trajecting tracking or interception engagement, finite-time SDRE for F16 aircraft flight controls \cite{cpws:2022}, SDRE optimal control design for quadrotors for enhancing the robustness against unmodeled disturbances \cite{chhu:2022}, and position/velocity controls for a high-speed vehicle \cite{abbe:1999a}. 
However, these application problems are real-world manipulation, the stochastic SDRE (SSDRE) should indispensably be considered.
That is, the goal in SSDRE control is to minimize the cost function \eqref{eq:cost_fun} and compute the optimal control $u$ at each fixed/frozen state $x$.

Assume that, with $A_0 \equiv A \equiv A(x)$, $\{ A_0^i \equiv A_0^i(x), B_0^i \equiv B_0^i(x) \}_{i=1}^r$, $Q \equiv Q(x)$, $R \equiv R(x)$, $L \equiv L(x)$,
\begin{description}
    \item{(c1)} The pair $(\{ A_0^i \}_{i=0}^r, \{ B_0^i \}_{i=0}^r)$ is stabilizable, i.e., there exists $F \in \mathbb{R}^{m \times n}$ such that the linear differential equation
    \begin{align*}
        \frac{d}{dt} Z(t) = \mathcal{L}_F Z(t) \equiv (A + BF)^{\top} Z + Z (A + BF) + \sum_{i=1}^r(A_0^i + B_0^iF)^{\top} Z(A_0^i + B_0^iF)
    \end{align*}
    is exponentially stable, i.e., $e^{\mathcal{L}_FZ(t-t_0)}$ is exponentially stable; 
    \item{(c2)} The pair $( \{ A_0^i \}_{i=0}^r, C)$ is detectable with $C^{\top} C = Q - L R^{-1} L^{\top}$, that is, $( \{ A_0^{i\top} \}_{i=0}^r, \{ C_i \}_{i=0}^r)$ is stabilizable with $C_0 = C$ and $\{ C_i \equiv 0 \}_{i=1}^r$.
\end{description}

For a fixed state $x$, if both (c1) and (c2) hold, then the stochastic continuous-time algebraic Riccati equation (SCARE) corresponding to \eqref{eq:dyna_SSDC} and \eqref{eq:cost_fun}
%
%The stochastic continuous-time algebraic Riccati equation (SCARE), in general, has the following form:
\begin{align}
     & A^{\top} X + X A + \sum_{i=1}^r A_0^{i\top} X A_0^i + Q \nonumber \\
     & - \left( X B + \sum_{i=1}^r A_0^{i\top} X B_0^i + L \right) 
      \left( \sum_{i=1}^r B_0^{i\top} X B_0^i + R\right)^{-1} 
     \left( B^{\top} X + \sum_{i=1}^r B_0^{i\top} X A_0^i + L^{\top} \right) = 0, \label{eq:SCARE_gen}
\end{align}
has a unique positive semi-definite (PSD) stabilizing solution $X_* \equiv X_*(x)$ \cite{drms:2013} such that the system $(A(x) + B(x) F_{X_*},A_0^1(x) + B_0^1(x) F_{X_*}, \cdots, A_0^r(x) + B_0^r(x) F_{X_*})$ is stable with 
\begin{align}
    F_{X_*} = - \left( \sum_{i=1}^r B_0^i(x)^{\top} X_* B_0^i(x) + R(x) \right)^{-1} \left( B(x)^{\top} X_* + \sum_{i=1}^r B_0^i(x)^{\top} X_* A_0^i(x) + L(x)^{\top}\right). \label{eq:F_X*}
\end{align}
In fact, $X_*$ is a stabilizing solution if and only if the closed-loop system
\begin{align*}
    d x(t) = ( A(x) + B(x) F_{X_*}) x + \sum_{i=1}^r ( A_0^i(x) x + B_0^i(x) F_{X_*}) d w_i(t)
\end{align*}
is exponentially stable. 

For convenience for discussion of convergence, the SCARE \eqref{eq:SCARE_gen} can be rewritten as 
\begin{multline}\label{eq:Stoc-ARE-N}
A^{\top}X+XA+Q+\Pi_{11}(X) \\
  -(XB+L+\Pi_{12}(X))(R+\Pi_{22}(X))^{-1}(XB+L+\Pi_{12}(X))^{\top}=0,
\end{multline}
where
\begin{align}\label{Pi}
\Pi(X)&\equiv \begin{bmatrix}
              \Pi_{11}(X) & \Pi_{12}(X) \\
              \Pi_{12}(X)^{\top} & \Pi_{22}(X)
            \end{bmatrix}
\equiv \begin{bmatrix}
              \sum_{i=1}^r A_0^{i\top}XA_0^i & \sum_{i=1}^r A_0^{i\top}XB_0^i \\
              \sum_{i=1}^r B_0^{i\top}XA_0^i & \sum_{i=1}^r B_0^{i\top}XB_0^i
            \end{bmatrix}\nonumber\\
         &=\begin{bmatrix}
              A_0^{1\top} &\cdots&A_0^{r\top}  \\
              B_0^{1\top} &\cdots&B_0^{r\top}  \\
            \end{bmatrix}(I_r\otimes X)  \begin{bmatrix}
              A_0^1 &B_0^1  \\
              \vdots &\vdots  \\
              A_0^r &B_0^r  \\
            \end{bmatrix}.
\end{align}
Note that if $X\geq 0$, then $\Pi(X)\geq 0$ and it also holds that
\begin{align}
      \Pi(X) \geq \Pi(Y) \geq 0\ \mbox{ for } \ X \geq Y \geq 0. \label{eq:assumption_Pi}
\end{align}

The main task of this paper is to develop an efficient, reliable, and robust algorithm for solving SCARE in \eqref{eq:SCARE_gen} or \eqref{eq:Stoc-ARE-N}. 
% It is observed that the SCARE \eqref{eq:SCARE_gen} is a mixed type of the classical CARE
% \begin{align}
%     A^{\top} X + X A - X B R^{-1} B^{\top} X + Q = 0 \label{eq:classical_CARE}
% \end{align}
% and the classical discrete-time ARE (DARE) of the form
% \begin{align}
%     X = A_0^{i\top} X A_0^i - A_0^{i\top} X B_0^i ( R+B_0^{i\top} X B_0^i)^{-1} B_0^{i\top} X A_0^i + Q, \ i = 1, \ldots, r. \label{eq:classical_DARE}
% \end{align}
Several numerical methods, such as the fixed-point (FP) method \cite{guli:2023}, Newton (NT) method \cite{dahi:2001}, modified Newton (mNT) method \cite{chllw:2011,guo:2002a,ivan:2007}, structure-preserving double algorithms (SDAs) \cite{guli:2023}, have been proposed for solving the positive semi-definite (PSD) solutions for SCARE in \eqref{eq:SCARE_gen}. 
In fact, based on the FP, NT, mNT and SDA methods, we can combine these methods to propose various algorithms to efficiently solve SCARE \eqref{eq:SCARE_gen} or \eqref{eq:Stoc-ARE-N}. 
Nevertheless, factors such as the conditions for convergence, monotonic convergence, convergence speed, accuracy of residuals and computational cost of iterative algorithms can ultimately determine the effectiveness, reliability and robustness of a novel algorithm.

We consider the following four algorithms for solving SCARE.
\begin{description}
    \item{(i)} FP-CARE\_SDA: We rewrite \eqref{eq:Stoc-ARE-N} as %in \eqref{eq:SARE_CARE_0} as 
\begin{align}
      \mathcal{R}(X) \equiv A_c(X)^{\top} X + X A_c(X) - X G_c(X) X + H_c(X) = 0 \label{eq:SARE_CARE_0}
\end{align}
where $A_c(X), G_c(X)$ and $H_c(X)$ are defined in \eqref{eq:mtx_Ac}, \eqref{eq:mtx_Gc} and \eqref{eq:mtx_Hc}, respectively.  We frozen $X$ in $A_c(X), G_c(X)$ and $H_c(X)$ as a fixed-point iteration and solve the CARE \eqref{eq:SARE_CARE_0} by SDA \cite{hull:2018,lixu:2006}, called the FP-CARE\_SDA.

Let $\widehat{X} \geq 0$ be a PSD solution of SCARE \eqref{eq:Stoc-ARE-N}. The sequence $\{X_k\}_{k=0}^{\infty}$ generated by FP-CARE\_SDA with $X_0=0$ is monotonically nondecreasing and R-linearly convergent to a PSD $0\le \widehat{X}_-\le \widehat{X}$ (see Theorem~\ref{thm:monotonical_increasing}). On the other hand, if $X_0\ge \widehat{X}$ such that $A_c(X_0)-G_c(X_0)X_0$ is stable, and $\mathcal{R}(X_0)\le 0$, FP-CARE\_SDA generates a monotonically nonincreasing sequence which converges R-linearly to a PSD $\widehat{X}_+\ge\widehat{X}\ge 0$ (see Theorem~\ref{thm:monotonical_decreasing}).

\item{(ii)} NT-FP-Lyap\_SDA: The SCARE \eqref{eq:Stoc-ARE-N} can be regarded as a nonlinear equation and solved by Newton's method in \eqref{eq:equ_Newton_iter_0} which has been derived by \cite{dahi:2001}
\begin{align}
    A_{X_k}^{\top} X + X A_{X_k} + \Pi_{X_k}(X) = - M_{X_k}, \label{eq:equ_Newton_iter_0}
\end{align}
where $A_{X_k}$, $\Pi_{X_k}(X)$ and $M_{X_k}$ are given in \eqref{eq:NT-AXk}, \eqref{eq:NT-PiX} and \eqref{eq:NT-MXk}, respectively. To solve the Newton's step \eqref{eq:equ_Newton_iter_0}, we frozen $X$ in $\Pi_{X_k}(X)$ as a fixed-point iteration and solve the associated Lyapunov equation by the special L-SDA in Algorithm~\ref{alg:Lyapunov_SDA}.

Let $\widehat{X} \geq 0$ be a PSD solution of SCARE \eqref{eq:Stoc-ARE-N}. If there is a $X_0 \geq \widehat{X}$ such that $\mathcal{R}_{X_0}'$ is stable, then $\{ X_k \}_{k=0}^{\infty}$ generated by NT-FP-Lyap\_SDA is monotonically nonincreasing and quadratically convergent to a PSD $\widehat{X}_+\ge \widehat{X}$ (see \cite{dahi:2001}). Here $\mathcal{R}_{X_0}'$ denotes the Fr\`{e}chet partial derivative with respect to $X_0$.
\item{(iii)} mNT-FP-Lyap\_SDA: The method NT-FP-Lyap\_SDA proposed in (ii) can be accelerated by the modified NT (mNT) step \cite{guo:2002a}. That is, the Lyapunov equation \eqref{eq:equ_Newton_iter_0} with the frozen term $\Pi_{X_k}(X)$ only solved once for approximating the modified Newton step:
\begin{align*}
    A_{X_k}^{\top} X + X A_{X_k} = - \Pi_{X_k}(X_k) - M_{X_k}. 
\end{align*}
If $X_0\ge \widehat{X}$ with $\mathcal{R}(X_0)\le 0$ such that $A_0-G_0X_0$ is stable, and $X_k\ge 0$ for all $k\ge 0$ generated by mNT-FP-Lyap\_SDA algorithm, $\{X_k\}_{k=1}^{\infty}$ is monotonically nonincreasing and converges to a PSD $\widehat{X}_+\ge \widehat{X}\ge 0$ (see Theorem~\ref{thm:monotonical_decreasing_mNT}).
\item{(iv)} FP: In \cite{guli:2023}, the authors used the M\"{o}bius transformation to transform the SCARE \eqref{eq:Stoc-ARE-N} into a stochastic discrete-time algebraic Riccati equation (SDARE) and proposed following fixed-point (FP) iteration to solve SDARE:
\begin{align*}
     X_1 & = H_{\gamma}, \nonumber \\
     X_{k+1} &= E_{\gamma}^{\top}(X_k \otimes I_{r+1}) \left(I_{n(r+1)} + G_{\gamma}(X_k \otimes I_{r+1})\right)^{-1} E_{\gamma} + H_{\gamma}, %\label{eq:FP-SCARE}
\end{align*}
where $E_{\gamma}$, $G_{\gamma}$, and $H_{\gamma}$ are defined in \eqref{eq:mtx_FP}.  The convergence of the fixed-point iteration has been proved in \cite{guli:2023}.  
\end{description}

This paper is organized as follows. We propose the FP-CARE\_SDA (FPC) algorithm for solving SCARE (1.3) in Section~\ref{sec:FP-CARE-SDA} and prove its monotonically nondecreasing/nonincreasing convergence with suitable initial matrices in Section~\ref{sec:conv-FP-CARE-SDA}. 
We propose the mNT-FP-Lyap\_SDA algorithm and prove its monotonic convergences in Sections~\ref{sec:Newton_mNewton} and \ref{sec:conv_mNewton}, respectively. 
Moreover, for the practical applications, we proposed FPC-NT-FP-Lyap\_SDA and FPC-mNT-FP-Lyap\_SDA algorithms, which are combined FP-CARE\_SDA with NT-FP-Lyap\_SDA and mNT-FP-Lyap\_SDA algorithms.
%respectively, in Section~\ref{sec:Newton_mNewton}. 
In Section~\ref{sec:numerical}, numerical examples from benchmarks show the convergence behavior of proposed algorithms. 
The real-world practical applications demonstrate the robustness of FP-CARE\_SDA and the efficiency of the FPC-mNT-FP-Lyap\_SDA.

%In this paper, we mainly propose the FP-CARE\_SDA, NT-FP-Lyap\_SDA and mNT-FP-Lyap\_SDA for solving the SCARE \eqref{eq:Stoc-ARE-N}. From numerical and theoretical points of view, FP-CARE\_SDA outperforms the other two methods for providing a good initial matrix for convergence while the NT, mNT-FP-Lyap\_SDA fail to converge with some randomly chosen initial matrix.

% In this paper we mainly propose efficient and reliable numerical algorithms for solving stochastic continuous-time algebraic Riccati equations(SCARE) typically arising from the differential state-dependent Riccati equation technique from the 3D missile/target engagement, the F16 aircraft flight control and the quadrotor optimal control etc. 
% To this end, we develop a fixed point (FP)-type iteration with solving a CARE by the structure-preserving doubling algorithm (SDA) at each iterative step, called FP-CARE\_SDA. We prove that either the FP-CARE\_SDA is monotonically non-decreasing and is R-linearly convergent with the zero initial matrix, or is monotonically and non-increasing convergent with a special initial matrix satisfying the convergence condition. 
% The FP-CARE\_SDA (FPC) algorithm can be regarded as a robust initial step to produce a good initial matrix, and then the modified Newton (mNT) method can be used by solving the corresponding Lyapunov equation with SDA. Numerical experiments show that the FPC-mNT-Lyap\_SDA algorithm outperforms the other existing algorithms.

\section{Fixed-Point CARE\_SDA method} \label{sec:FP-CARE-SDA}
%\section{Fixed-point method}
 In this section, we propose a FP-CARE\_SDA algorithm to solve \eqref{eq:Stoc-ARE-N}.
Let
\begin{align}
     L_c(X) = L +  \Pi_{12}(X), \quad R_c(X) = R+\Pi_{22}(X), \quad Q_c(X) = Q+\Pi_{11}(X). \label{eq:mtx_LRQ_c}
\end{align}
Then equation \eqref{eq:Stoc-ARE-N} is equivalent to
\begin{align*}
      A^{\top} X + X A- (XB + L_c(X)) R_c(X)^{-1} (XB + L_c(X))^{\top} + Q_c(X) = 0
\end{align*}
or
\begin{subequations}\label{eq1.2}
\begin{align}
      A_c(X)^{\top} X + X A_c(X) - X G_c(X) X + H_c(X) = 0 \label{eq:SARE_CARE}
\end{align}
with 
\begin{align}
      A_c(X) &= A - BR_c(X)^{-1} L_c(X)^{\top}, \label{eq:mtx_Ac}\\ 
      G_c(X) &= B R_c(X)^{-1} B^{\top}, \label{eq:mtx_Gc}\\ 
      H_c(X) &= Q_c(X) - L_c(X) R_c(X)^{-1} L_c(X)^{\top}. \label{eq:mtx_Hc}
\end{align}
\end{subequations}
For a given $X_k\ge 0$, we  denote
\begin{align}
      R_k = R_c(X_k), \quad G_k = G_c(X_k), \quad A_k = A_c(X_k), \quad H_k = H_c(X_k) \label{eq:mtx_AGHR_k}
\end{align}
and consider the following CARE
\begin{align}
      A_k^{\top} X + X A_k - X G_k X + H_k = 0, \label{eq:kth-SARE_CARE}
\end{align}
where $G_k\ge 0$ %, $H_k\ge 0$ 
 from \eqref{eq:mtx_AGHR_k}. % and \eqref{eq:assumption_Pi}.

\begin{proposition} \label{prop:stabilizable}
      If $(A,B)$ is stabilizable, $(A_k, G_k)$ is stabilizable.
\end{proposition}
\begin{proof}
     Let $y^{\top} ( A - B R_k^{-1} L_c(X_k)) = \lambda y^{\top}$ with ${\rm Re}(\lambda) \geq 0$. Suppose $y^{\top}G_k = y^{\top} B R_k^{-1} B^{\top} = 0$. This implies that $y^{\top} B R_k^{-1/2} = 0$, i.e., $y^{\top} B = 0$.  Then we have $y^{\top} A = \lambda y^{\top}$, ${\rm Re}(\lambda) \geq 0$, which implies that $y = 0$ by stabilzability of $(A, B)$.
\end{proof}
\begin{proposition} \label{prop:detectable}
     Let $C^{\top} C= Q - L R^{-1} L^{\top}$ (a full rank decomposition). Suppose that % or 
     ${\rm Ker}C \subseteq \left(\bigcap_{i=1}^r{\rm Ker} A_0^i\right)\bigcap {\rm Ker}L$. If $(C, A)$ is detectable, $(H_k, A_k)$ is detectable and $H_k \geq 0$ for $X_k \geq 0$.
\end{proposition}
\begin{proof}
     Let $A_k x = A x - BR_k^{-1} (L+\sum_{i=1}^r (B_0^i)^{\top} X_k A_0^i) x = \lambda x$ with ${\rm Re}(\lambda) \geq 0$. If $H_k x = 0$, then from \eqref{eq:mtx_AGHR_k} and \eqref{eq:mtx_Hc} follows that 
     \begin{align}\label{eq1.4.1}
     x^{\top} C^{\top} C x + x^{\top} \Lambda(X_k) x = 0, 
     \end{align}
     where $\Lambda(X) \equiv (Q-C^{\top} C+\Pi_{11}(X)) - L_c(X) R_c(X)^{-1}  L_c(X)^{\top}$. From \eqref{Pi}, it holds that
     \begin{align*}
            0 \leq \Pi(X_k) + \begin{bmatrix}
                 Q-C^{\top} C & L \\ L^{\top} & R
            \end{bmatrix} = \begin{bmatrix}
                 I & L_c(X_k) R_c(X_k)^{-1} \\ 0 & I
            \end{bmatrix} \begin{bmatrix} 
                 \Lambda(X_k) & 0 \\ 0 & R_c(X_k)
            \end{bmatrix} \begin{bmatrix}
                 I & 0 \\ R_c(X_k)^{-1} L_c(X_k)^{\top} & I
            \end{bmatrix},
     \end{align*}
     which implies that $\Lambda(X_k) \geq 0$   and $H_k = C^{\top} C + \Lambda(X_k) \geq 0$.
     Hence, we have $C x = 0$ and $x^{\top} \Lambda(X_k) x = 0$. Since ${\rm Ker}C \subseteq \left(\bigcap_{i=1}^r{\rm Ker} A_0^i\right)\bigcap {\rm Ker}L$, we obtain that $A_0^i x = 0$ for $i = 1, \ldots, r$. This means that $A x = \lambda x$ with ${\rm Re}(\lambda) \geq 0$. Therefore, $x = 0$ by the detectability of $(C,A)$ and then $(H_k, A_k)$ is detectable.
%     \begin{align*}
%         x^{\top} \Lambda(X_k) x %&= x^{\top} \left(\Pi_{11}(X_k) - \Pi_{12}(X_k) R_k^{-1}  \Pi_{12}(X_k)^{\top} \right) x \\
%          &=  \sum_{i=1}^r x^{\top} A_i^{\top}X_k A_i x - x^{\top}\sum_{i=1}^r A_i^{\top}X_k B_i  R_k^{-1} \sum_{j=1}^r B_j^{\top}X_kA_j x \\
%          %&= \sum_{i=1}^r (A_i x)^{\top} X_k (A_i x) - \left(\sum_{i=1}^r B_i^{\top}X_kA_i x \right)^{\top} R_k^{-1} \left(\sum_{i=1}^r B_i^{\top}X_kA_i x \right) \\
%          &= x^{\top} \begin{bmatrix} A_1^{\top} & \cdots & A_r^{\top} \end{bmatrix} (I_r \otimes X_k) \begin{bmatrix}
%               A_1 \\ \vdots \\ A_r
%          \end{bmatrix} x \\
%          &- x^{\top} \begin{bmatrix} A_1^{\top} & \cdots & A_r^{\top} \end{bmatrix} \begin{bmatrix}
%               X_k B_1 \\ \vdots \\ X_k B_r
%          \end{bmatrix} R^{-1}_k \begin{bmatrix}
%               B_1 X & \cdots & B_r X_k
%          \end{bmatrix}\begin{bmatrix}
%               A_1 \\ \vdots \\ A_r
%          \end{bmatrix} x \\
%          &= (\begin{bmatrix}
%               A_1 \\ \vdots \\ A_r
%          \end{bmatrix} x)^{\top} \left\{ I_r \otimes X_k - \begin{bmatrix}
%               X_k B_1 \\ \vdots \\ X_k B_r
%          \end{bmatrix} R^{-1}_k \begin{bmatrix}
%               B_1 X & \cdots & B_r X_k
%          \end{bmatrix} \right\} \begin{bmatrix}
%               A_1 \\ \vdots \\ A_r
%          \end{bmatrix} x
%     \end{align*}
\end{proof}

\begin{theorem}\label{thm2.3}
    For a given $X_k\ge 0$, the CARE of \eqref{eq:kth-SARE_CARE} has a unique PSD stabilizing solution $X_{k+1}\ge 0$.
\end{theorem}
\begin{proof}
    From Propositions~\ref{prop:stabilizable} and \ref{prop:detectable}, as well as the uniqueness and existence theorem \cite{laro:1980} follows that there is a unique stabilizing solution $X_{k+1}\ge 0$ for \eqref{eq:kth-SARE_CARE}.
\end{proof}

 It has been shown in \cite{hull:2018,lixu:2006} that the structure-preserving doubling algorithm (SDA) can be applied to solve \eqref{eq:kth-SARE_CARE} with quadratic convergence. In this paper, 
we consider using the SDA as an implicit fixed-point iteration, called FP-CARE\_SDA as stated in Algorithm~\ref{alg:SDA-CARE}, for solving the SCARE~\eqref{eq:Stoc-ARE-N}. 
%by solving $X_{k+1} \geq 0$ for CARE of \eqref{eq:kth-SARE_CARE}.
%\begin{align}
%      A_k^{\top} X_{k+1} + X_{k+1} A_k - X_{k+1} G_k X_{k+1} + H_k = 0. \label{eq:kth-SARE_CARE}
%\end{align}

\begin{algorithm}
\caption{FP-CARE\_SDA method for solving SCARE \eqref{eq:Stoc-ARE-N} } \label{alg:SDA-CARE}
\begin{algorithmic}[1]
\REQUIRE $A, Q \in \mathbb{R}^{n \times n}$, $B, L  \in \mathbb{R}^{n \times m}$, $R = R^{\top} \in \mathbb{R}^{m \times m}$, $A_0^i \in \mathbb{R}^{n \times n}$, $B_0^i \in \mathbb{R}^{n \times m}$ for $i = 1, \ldots, r$, and  a tolerance $\varepsilon$.
\ENSURE Solution $X$.
\STATE Set $k = 0$ and $\delta = \infty$.
\STATE Choose an initial $X_0^{\top} = X_0$.
\WHILE{$\delta>\varepsilon$}
\STATE Set $L_c = L + \Pi_{12}(X_k)$,  $R_c = \Pi_{22}(X_k) + R$ and $Q_c = \Pi_{11}(X_k) + Q$.
\STATE Set $A_c = A - B R_c^{-1} L_c^{\top}$, $G_c = B R_c^{-1} B^{\top}$ and $H_c = Q_c - L_c R_c^{-1} L_c^{\top}$.
\STATE Use SDA to solve the stabilizing solution $X_{k+1}$ of the CARE \label{alg:step6_FP-CARE}
\begin{align*}
      A_c^{\top} X + X A_c  - X G_c  X + H_c  = 0.
\end{align*} 
\STATE Compute the normalized residual $\delta$ of \eqref{eq:Stoc-ARE-N} with $X = X_{k+1}$. 
\STATE Set $k = k + 1$.
\ENDWHILE 
\end{algorithmic}
\end{algorithm}

\section{Convergence Analysis for FP-CARE\_SDA Method} \label{sec:conv-FP-CARE-SDA}
In this section, we will study the convergence analysis of the FP-CARE\_SDA algorithm. Let $\widehat{X}$ be a PSD solution of SCARE \eqref{eq:SCARE_gen}. With different initial conditions of $X_0$, the sequence $\{X_k\}_{k=0}^{\infty}$ generated by FP-CARE\_SDA is monotonically nondecreasing/nonincreasing to $\widehat{X}$ and R-linearly convergent to some PSD solution $\widehat{X}_-/\widehat{X}_+$, respectively, with $\widehat{X}_+\ge \widehat{X}\ge \widehat{X}_-\ge 0$.

Let 
\begin{align}
    \Gamma(X)&\equiv \left[\begin{array}{cc|c}
              0&-A&-B\\
              -A^{\top}&Q_c(X) & L_{c}(X) \\  \hline 
              -B^{\top}&L_{c}(X)^{\top} & R_c(X)
            \end{array}\right] \equiv \left[\begin{array}{c|c}
         \Gamma_{11}(X) & \Gamma_{12}(X) \\ \hline\Gamma_{21}(X)^{\top} & R_c(X)
    \end{array}\right]  \in \mathbb{R}^{(2n+m)\times (2n+m)}, \label{eq:mtx_Gamma}
\end{align}
where $Q_c(X)$, $L_c(X)$, and $R_c(X)$ are given in \eqref{eq:mtx_LRQ_c}. 
Then the Schur complement of the block matrix $\Gamma(X)$ has the form
\begin{align}\label{eq:mtx_Omega}
    \Omega(X)=&\begin{bmatrix}
        0&-A\\
     -A^{\top}&Q_c(X)    \end{bmatrix}-
            \begin{bmatrix}
              -B\\
            L_{c}(X)        \end{bmatrix}R_c(X)^{-1}\begin{bmatrix}
              -B^{\top}&L_{c}(X)^{\top}  
            \end{bmatrix}\nonumber\\
   =& \begin{bmatrix}
        -BR_c(X)^{-1}B^{\top}&-(A-BR_c(X)^{-1} L_{c}(X)^{\top})\\
     -(A^{\top}- L_{c}(X)R_c(X)^{-1}B^{\top})&Q_c(X)- L_{c}(X)R_c(X)^{-1}L_{c}(X)^{\top}  \end{bmatrix} \nonumber\\
     =&\begin{bmatrix}
        -G_c(X)&-A_c(X)\\
     -A_c(X)^{\top}&H_c(X)  \end{bmatrix},
\end{align}
where $A_c(X)$, $G_c(X)$ and $H_c(X)$ in \eqref{eq1.2} are the coefficient matrices of CARE \eqref{eq:SARE_CARE}. 
% Suppose that $X_{k-1}\geq 0$. 
% % and denote
% % \begin{align}\label{eq2.8}
% %     A_{k-1}=A_c(X_{k-1}), \ \ G_{k-1}=G_c(X_{k-1})\text{ and }H_{k-1}=H_c(X_{k-1}).
% % \end{align}
% From Theorem \ref{thm2.3}, it holds that the CARE
% \begin{align}
%       A_{k-1}^{\top} X_{k} + X_{k} A_{k-1} - X_{k} G_{k-1} X_{k} + H_{k-1} = 0, \label{eq:k-1th-SARE_CARE}
% \end{align}
% has a PSD solution $X_k\ge 0$.

\begin{lemma} \label{lem:increasing_Omega}
Let $\Omega(X)$ be defined in \eqref{eq:mtx_Omega}. Then
\begin{align}\label{eq:increasing_Omega}
     \Omega(X)\ge \Omega(Y)\text{ for } X\ge Y\ge 0. 
\end{align}
\end{lemma}
\begin{proof}
Suppose that $X\ge Y\ge 0$. From \eqref{eq:mtx_Gamma}, \eqref{eq:mtx_LRQ_c} and \eqref{Pi}, we have
\begin{align*}
    \Gamma(X)&=\left[\begin{array}{c|cc}
              0&-A&-B\\\hline
              -A^{\top}&Q & L \\   
              -B^{\top}&L^{\top} & R
    \end{array}\right]+\left[\begin{array}{ccc}
              0&\cdots&0\\\hline 
              A_0^{1\top}&\cdots & A_{0}^{r\top} \\  
              B_0^{1\top}&\cdots & B_{0}^{r\top}
            \end{array}\right] (I_r\otimes X)\left[\begin{array}{c|cc}
              0&A_0^1&B_0^1\\
              \vdots&\vdots & \vdots \\  
              0&A_0^r & B_0^{r}
            \end{array}\right]\\
            &\ge\left[\begin{array}{c|cc}
              0&-A&-B\\\hline
              -A^{\top}&Q & L \\   
              -B^{\top}&L^{\top} & R
    \end{array}\right]+\left[\begin{array}{ccc}
              0&\cdots&0\\\hline 
              A_0^{1\top}&\cdots & A_{0}^{r\top} \\  
              B_0^{1\top}&\cdots & B_{0}^{r\top}
            \end{array}\right] (I_r\otimes Y)\left[\begin{array}{c|cc}
              0&A_0^1&B_0^1\\
              \vdots&\vdots & \vdots \\  
              0&A_0^r & B_0^{r}
            \end{array}\right]=\Gamma(Y).
\end{align*}
Now, we show that \eqref{eq:increasing_Omega}. 
For any $\bw\in \mathbb{R}^{2n}$, since $\Gamma(X)\ge \Gamma(Y)$, we have
    \begin{align}\label{eq:Gamma_w_increasing}
        \Gamma_{\bw}(X)\equiv\left[\begin{array}{cc}
        \bw^{\top}\Gamma_{11}(X)\bw &\bw^{\top}\Gamma_{12}(X)   \\
        \Gamma_{12}(X)^{\top}\bw  & R_c(X) 
    \end{array}\right]\ge \left[\begin{array}{cc}
        \bw^{\top}\Gamma_{11}(Y)\bw &\bw^{\top}\Gamma_{12}(Y)   \\
        \Gamma_{12}^{\top}(Y)\bw  & R_c(Y) 
    \end{array}\right]\equiv \Gamma_{\bw}(Y).
    \end{align}
    It is easily seen that $\bw^{\top}\Omega(X)\bw$ and $\bw^{\top}\Omega(Y)\bw$ are Schur complement of $\Gamma_{\bw}(X)$ and $\Gamma_{\bw}(Y)$, respectively. Suppose that $\alpha\in \mathbb{R}_{+}$ such that 
    \begin{align*}
        \bw^{\top}\Omega(Y)\bw+\alpha>0.
    \end{align*}
Let
    \begin{align}\label{eq:Gamma_wa}
        \Gamma_{\bw}^{\alpha}(X)=\Gamma_{\bw}(X)+\left[\begin{array}{cc}
            \alpha &0  \\
            0 & 0
        \end{array}\right],\ \ \Gamma_{\bw}^{\alpha}(Y)=\Gamma_{\bw}(Y)+\left[\begin{array}{cc}
            \alpha &0  \\
            0 & 0
        \end{array}\right].
    \end{align}
    Note that $\bw^{\top}\Omega(Y)\bw+\alpha$ is the Schur complement of $\Gamma_{\bw}^{\alpha}(Y)$. Since $R_c(Y)>0$ and $\bw^{\top}\Omega(Y)\bw+\alpha>0$, it follows from \eqref{eq:Gamma_w_increasing} that $\Gamma_{\bw}^{\alpha}(X)\ge \Gamma_{\bw}^{\alpha}(Y)>0$. Hence,
    \begin{align}\label{eq:Gamma_wa_inv}
        0<\Gamma_{\bw}^{\alpha}(X)^{-1}\le \Gamma_{\bw}^{\alpha}(Y)^{-1}.
    \end{align}
    From \eqref{eq:Gamma_wa}, we obtain that the $(1,1)$ entries of $\Gamma_{\bw}^{\alpha}(X)^{-1}$ and $\Gamma_{\bw}^{\alpha}(Y)^{-1}$ are $(\bw^{\top}\Omega(X)\bw+\alpha)^{-1}$ and $(\bw^{\top}\Omega(Y)\bw+\alpha)^{-1}$, respectively. It follows from \eqref{eq:Gamma_wa_inv} that $0<(\bw^{\top}\Omega(X)\bw+\alpha)^{-1}\le (\bw^{\top}\Omega(Y)\bw+\alpha)^{-1}$ for any $\bw \in \mathbb{R}^{2n}$. This implies that \eqref{eq:increasing_Omega} holds.
\end{proof}

\begin{remark} \label{rem:equiv_form_SCARE}
Let $\Omega(X)$ be defined in \eqref{eq:mtx_Omega}. The SCARE in \eqref{eq:SARE_CARE} can be represented as  
\begin{align}\label{eq:new_rep_SCARE}
    \begin{bmatrix} X &  -I \end{bmatrix} \Omega(X)\begin{bmatrix}
        X\\
     -I  \end{bmatrix}=0.
\end{align}
% From \eqref{eq:mtx_LRQ_c}, \eqref{eq:mtx_Omega}, and Lemma \ref{lem:increasing_Omega}, if $X\ge 0$, then we have
% \begin{align*}
%     \begin{bmatrix}
%         -G_c(X)&-A_c(X)\\
%      -A_c(X)^{\top}&H_c(X)  \end{bmatrix}\ge \begin{bmatrix}
%         -BR^{-1}B^{\top}&-(A-BR^{-1}L^{\top})\\
%      -(A^{\top}-LR^{-1}B^{\top})&Q-LR^{-1}L^{\top}  \end{bmatrix}.
% \end{align*}
% From the assumption (c2), \eqref{eq:mtx_Gc} and \eqref{eq:mtx_AGHR_k}, we obtain that if $X_k\ge 0$, then $G_k, H_k\ge 0$.
\end{remark}

Let $\{X_k\}_{k=0}^{\infty}$ be a sequence generated by Algorithm 1. 
Suppose that $X_{k-1}\geq 0$. 
% and denote
% \begin{align}\label{eq2.8}
%     A_{k-1}=A_c(X_{k-1}), \ \ G_{k-1}=G_c(X_{k-1})\text{ and }H_{k-1}=H_c(X_{k-1}).
% \end{align}
From Theorem \ref{thm2.3}, it holds that the CARE
\begin{align}
      A_{k-1}^{\top} X_{k} + X_{k} A_{k-1} - X_{k} G_{k-1} X_{k} + H_{k-1} = 0, \label{eq:k-1th-SARE_CARE}
\end{align}
has a PSD solution $X_k\ge 0$.
%Then for each $k$, $X_k$ is the stabilizing solution of \eqref{eq:k-1th-SARE_CARE}. 
That is,
\begin{align}\label{eq2.13}
    \begin{bmatrix} X_{k} &  -I \end{bmatrix} \Omega(X_{k-1})\begin{bmatrix}
        X_k\\
     -I  \end{bmatrix}=0.
\end{align} 
%Using the above lemma, we can prove the following monotonic property of $\{ X_k \}_{k=0}^{\infty}$ under some conditions.
We now state a well-known result \cite{laro:1995} to show the monotonicity of $\{ X_k \}_{k=0}^{\infty}$.
\begin{lemma}{\cite{laro:1995}} \label{lem3.2}
    Let $A,Q\in \mathbb{R}^{n\times n}$ with $A$ stable and $Q\le 0$. Then the solution of the Lyapunov equation $A^{\top}X+XA=Q$ is PSD.
\end{lemma}
% \begin{lemma}\label{lem3.3}
%     Let $A,Q,H\in \mathbb{R}^{n\times n}$ with $Q\le 0$ and $H\le 0$. Assume the Lyapunov equation $A^{\top}X+XA= Q+H$ has a unique semi-positive definite solution. If $(H,A)$ is detectable, then $A$ is stable.
% \end{lemma}
%In the following theorem, we show that $\{ X_k \}_{k=0}^{\infty}$ is monotonically nondecreasing if the initial $X_0=0$.
\begin{theorem} \label{thm:monotonical_increasing}
      Let $\widehat{X} \ge 0$ be a solution of SCARE \eqref{eq:Stoc-ARE-N}. Suppose that $\{ X_k \}_{k=0}^{\infty}$ is a sequence generated by Algorithm 1 with $X_0=0$. 
 Then 
 \begin{description}
     \item[(i)] $X_k\le X_{k+1}$, $X_k\le \widehat{X}$, $\mathcal{R}(X_k)\ge 0$, and $A_k-G_k\widehat{X}$ is stable for each $k\ge 0$.
     %\item[2.] $A_k-G_k\widehat{X}$ is stable for each $k\ge 0$.
     \item[(ii)] $\lim_{k\rightarrow \infty}X_k=\widehat{X}_-$, where $\widehat{X}_-\le \widehat{X}$ is a solution of SCARE \eqref{eq:Stoc-ARE-N}.
     \item[(iii)] $\sigma\left(A_c(\widehat{X}_-)-G_c(\widehat{X}_-)\widehat{X}_-\right)\subset \mathbb{C}_-\bigcup i\mathbb{R}$.
 \end{description}
      % \begin{align*}
      %     X_0\le X_1\le \cdots\le X_{k-1}\le X_k\le \cdots \le\widehat{X}
      % \end{align*}
      % is a monotonically nondecreasing sequence, and converges to $\widehat{X}$. Furthermore, we have $\mathcal{R}(X_k)\ge 0$ for each $k.$
\end{theorem}
\begin{proof}
    (i). We prove the assertion by induction for each $k\ge 0$. 
    % \begin{align}\label{eq3.15}
    %     X_k\le X_{k+1},\ X_k\le \widehat{X},\ \mathcal{R}(X_k)\ge 0, \text{ and }A_k-G_k\widehat{X}\text{ is stable.}
    % \end{align}
    For $k=0$, we have $0=X_0\le \widehat{X}$ and $\mathcal{R}(X_0)=H_0\ge 0$. Since $X_1$ is the stabilizing PSD solution of 
    \begin{align*}
    A_0^{\top}X+XA_0-XG_0X+H_0=0,
\end{align*}
    where $(A_0, G_0)$ is stabilizable and $(H_0,A_0)$ is detectable, we have $X_1\ge X_0=0$. Now, we claim that $A_0-G_0\widehat{X}$ is stable. From Lemma \ref{lem:increasing_Omega} and using the fact that $X_{0}\le \widehat{X}$, we have 
    \begin{align}\label{eq3.12}
        (A_0^{\top}-\widehat{X}G_0)\widehat{X}+\widehat{X}(A_0-G_0\widehat{X})&=\begin{bmatrix} \widehat{X} &  -I \end{bmatrix} \Omega(X_0)\begin{bmatrix}
        \widehat{X}\\
     -I  \end{bmatrix}-H_0-\widehat{X}^{\top}G_0\widehat{X}\nonumber\\
     &\le -H_0-\widehat{X}^{\top}G_0\widehat{X}.
    \end{align}
    Suppose that $A_0-G_0\widehat{X}$ is not stable. Then there exist $y\neq 0$ and Re$(\lambda)\ge 0$ such that $(A_0-G_0\widehat{X})y=\lambda y$. From \eqref{eq3.12}, we have
    \begin{align*}
        2{\rm Re}(\lambda)y^*\widehat{X}y\le -y^*H_0y-y^*\widehat{X}^{\top}G_0\widehat{X}y.
    \end{align*}
    Since $\widehat{X}, H_0, G_0\ge 0$ and Re$(\lambda)\ge 0$, we obtain that $H_0y=0$ and $G_0\widehat{X}y=0$. This implies that $A_0y=\lambda y$ which is a contradiction because of the detectabity of $(H_0,A_0)$. Hence, $A_0-G_0\widehat{X}$ is stable. We now assume that the assertion (i) is true for $k\ge 0$ and $X_{k+2}$ is the stabilizing solution of the CARE
\begin{align*}
 \begin{bmatrix} X &  -I \end{bmatrix} \Omega(X_{k+1})\begin{bmatrix}
        X\\
     -I  \end{bmatrix}=0.
\end{align*}
Then we have $A_{k+1} - G_{k+1} X_{k+2}$ is stable and 
\begin{align*}
    & ( A_{k+1}^{\top} - X_{k+2} G_{k+1})(X_{k+2} - X_{k+1}) + ( X_{k+2} - X_{k+1})( A_{k+1} - G_{k+1} X_{k+2})  \nonumber\\
    =& A_{k+1}^{\top}X_{k+2}+X_{k+2}A_{k+1}- X_{k+2} G_{k+1}X_{k+2}-(A_{k+1}^{\top}X_{k+1}+X_{k+1}A_{k+1}-X_{k+1} G_{k+1}X_{k+1}) \nonumber \\
    &- (X_{k+2} - X_{k+1}) G_{k+1} ( X_{k+2} - X_{k+1}) \nonumber \\
    =& -\mathcal{R}(X_{k+1})- (X_{k+2} - X_{k+1}) G_{k+1} ( X_{k+2} - X_{k+1}).%\nonumber \\
%     =&  - \begin{bmatrix} X_{k+1} & -I \end{bmatrix} \Omega(X_{k+1}) \begin{bmatrix}
%        X_{k+1}\\
%     -I  \end{bmatrix}- (X_{k+2} - X_{k+1}) G_{k+1} ( X_{k+2} - X_{k+1}).
\end{align*}
From Lemma \ref{lem:increasing_Omega} and using the fact that $X_{k}\le X_{k+1}$, we have
\begin{align*}
    0=\begin{bmatrix} X_{k+1} & -I \end{bmatrix} \Omega(X_{k}) \begin{bmatrix}
        X_{k+1}\\
     -I  \end{bmatrix}\le \begin{bmatrix} X_{k+1} & -I \end{bmatrix} \Omega(X_{k+1}) \begin{bmatrix}
        X_{k+1}\\
     -I  \end{bmatrix}=\mathcal{R}(X_{k+1}).
\end{align*}
Hence, $X_{k+1}\le X_{k+2}$ by Lemma \ref{lem3.2}. Using the fact that $X_k\le \widehat{X}$ and Lemma \ref{lem:increasing_Omega}, we have 
\begin{align*}
    &(A_k^{\top}-\widehat{X}G_k)(\widehat{X}-X_{k+1})+(\widehat{X}-X_{k+1})(A_k-G_k\widehat{X})\\
    &=\begin{bmatrix} \widehat{X} &  -I \end{bmatrix} \Omega(X_k)\begin{bmatrix}
        \widehat{X}\\
     -I  \end{bmatrix}-(\widehat{X}-X_{k+1})G_k(\widehat{X}-X_{k+1})
     \le -(\widehat{X}-X_{k+1})G_k(\widehat{X}-X_{k+1}).
\end{align*}
This leads to $X_{k+1}\le\widehat{X}$ because $A_k-G_k\widehat{X}$ is stable. Now, we claim that $A_{k+1}-G_{k+1}\widehat{X}$ is stable. After calculation, it holds that  
\begin{subequations}\label{eq3.16}
\begin{align}\label{eq3.16a}
    &(A_{k+1}^{\top}-\widehat{X}G_{k+1})(\widehat{X}-X_{k+1})+(\widehat{X}-X_{k+1})(A_{k+1}-G_{k+1}\widehat{X})\nonumber\\
    &=\begin{bmatrix} \widehat{X} &  -I \end{bmatrix} \Omega(X_{k+1})\begin{bmatrix}
        \widehat{X}\\
     -I  \end{bmatrix}-\begin{bmatrix} X_{k+1} &  -I \end{bmatrix} \Omega(X_{k+1})\begin{bmatrix}
        X_{k+1}\\
     -I  \end{bmatrix}-(\widehat{X}-X_{k+1})G_{k+1}(\widehat{X}-X_{k+1}).
\end{align}
Similarly, using the fact that $X_k\le X_{k+1}\le\widehat{X}$ and Lemma \ref{lem:increasing_Omega}, we have
\begin{align}\label{eq3.16b}
    \begin{bmatrix} \widehat{X} &  -I \end{bmatrix} \Omega(X_{k+1})\begin{bmatrix}
        \widehat{X}\\
     -I  \end{bmatrix}\le 0 \ \text{ and } \ -\begin{bmatrix} X_{k+1} &  -I \end{bmatrix} \Omega(X_{k+1})\begin{bmatrix}
        X_{k+1}\\
     -I  \end{bmatrix}\le 0.
\end{align}
\end{subequations}
Assume that $A_{k+1} - G_{k+1} \widehat{X}$ is not stable. Then there exist $y \neq 0$ and $\mbox{Re}(\lambda) \geq 0$ such that $(A_{k+1} - G_{k+1} \widehat{X}) y = \lambda y$. From \eqref{eq3.16} and $ X_{k+1}\le\widehat{X}$, we obtain that $0 \leq 2 \mbox{Re}(\lambda) y^{*} (\widehat{X}-X_{k+1}) y$ and hence,  
\begin{align*}
    0& =y^*\begin{bmatrix} \widehat{X} &  -I \end{bmatrix} \Omega(X_{k+1})\begin{bmatrix}
        \widehat{X}\\
     -I  \end{bmatrix}y\\
     &=y^*\left[(A_{k+1}^{\top} - \widehat{X}G_{k+1} )\widehat{X}+\widehat{X}(A_{k+1} - G_{k+1} \widehat{X})+\widehat{X} G_{k+1} \widehat{X}+H_{k+1}\right]y\\
     &=2 \mbox{Re}(\lambda) y^{*}\widehat{X}y+y^{*}\widehat{X} G_{k+1} \widehat{X}y+y^{*} H_{k+1} y.
\end{align*}
Since $\widehat{X}, H_{k+1}, G_{k+1} \geq 0$, we obtain that $H_{k+1} y = 0$ and $G_{k+1} \widehat{X} y = 0$. 
% \begin{align*}
%     H_{k+1} y = 0 \quad \mbox{ and } \quad G_{k+1} \widehat{X} y = 0.
% \end{align*}
This means that
\begin{align*}
    \lambda y = (A_{k+1} - G_{k+1} \widehat{X}) y = A_{k+1} y, \quad y \neq 0, \quad \mbox{Re}(\lambda) \geq 0, \quad \mbox{ and } \quad H_{k+1} y = 0,
\end{align*}
which contradicts the detectability of $(H_{k+1}, A_{k+1})$. Therefore, $A_{k+1} - G_{k+1} \widehat{X}$ is stable. The induction process is complete. 

(ii). Since the sequence $\{X_k\}$ is monotonically nondecreasing and bounded above by $\widehat{X}$, we have $\lim_{k\rightarrow \infty}X_k=\widehat{X}_-$,  where $\widehat{X}_-\le \widehat{X}$ is a solution of SCARE \eqref{eq:Stoc-ARE-N}. 

(iii). By (i), all $A_{k} - G_{k} \widehat{X}_-$ are stable, by continuity $\sigma\left(A_c(\widehat{X}_-)-G_c(\widehat{X}_-)\widehat{X}_-\right)\subset \mathbb{C}_-\bigcup i\mathbb{R}$. 
\end{proof}

From another perspective, we can prove that $\{ X_k \}_{k=0}^{\infty}$ is monotonically nonincreasing if the initial $X_0$ satisfies some strict conditions.
\begin{theorem} \label{thm:monotonical_decreasing}
      Let $\widehat{X} \ge 0$ be a solution of SCARE \eqref{eq:Stoc-ARE-N}. Suppose that $\{ X_k \}_{k=0}^{\infty}$ is a sequence generated by Algorithm 1.  Let $X_0\ge \widehat{X}$ such that $A_0-G_0X_0$ is stable and $\mathcal{R}(X_0) \leq 0$.
% \begin{align}\label{eq2.15}
%  \mathcal{R}(X_0)\equiv A_0^{\top}X_0+X_0A_0-X_0G_0X_0+H_0 \le 0,   
% \end{align}
% where $A_0=A_c(X_0)$, $G_0=G_c(X_0)$, and $H_0=H_c(X_0)$ are given in \eqref{eq2.8}.
 Then 
 \begin{description}
      \item[(i)] $X_k\ge X_{k+1}$, $X_k\ge \widehat{X}$, $\mathcal{R}(X_k)\le 0$ for each $k\ge 0$.
     \item[(ii)] $\lim_{k\rightarrow \infty}X_k=\widehat{X}_+$, where $\widehat{X}_+\ge \widehat{X}$ is a solution of SCARE \eqref{eq:Stoc-ARE-N}.
     \item[(iii)] $\sigma\left(A_c(\widehat{X}_+)-G_c(\widehat{X}_+)\widehat{X}_+\right)\subset \mathbb{C}_-\bigcup i\mathbb{R}$. 
 \end{description}
 % $\{ X_k \}_{k=0}^{\infty}$ satisfying 
 %      \begin{align*}
 %          X_0\ge X_1\ge \cdots\ge X_{k-1}\ge X_k\ge \cdots \ge\widehat{X}
 %      \end{align*}
 %      is a monotonically nonincreasing sequence, and converges to $\widehat{X}$. Furthermore, we have $\mathcal{R}(X_k)\le 0$ for each $k$.
      % Assume that there is an $X_0 {\color{red} \geq \widehat{X}} \geq 0$ with $(A_0, G_0)$ stabilizable and $(H_0, A_0)$ is detectable such that $A_0 - G_0 X_0$ is stable and $\Omega(X_0) \leq 0$. 
      %Then the SDA-CARE iteration \eqref{eq:kth-SARE_CARE} generate a monotonic sequence $\{ X_k \}_{k=0}^{\infty}$ such that
     % $X_k \geq X_{k+1} {\color{red} \geq \widehat{X}} \geq 0$ and $\lim_{k \to \infty} X_k = \widehat{X} \geq 0$ solves SCARE \eqref{eq:Stoc-ARE-N}.
\end{theorem}
\begin{proof}
    (i). We prove the assertion by mathematical induction on $k \geq 0$.
    First, we show that if $X_{k-1}\ge \widehat{X}$,   $X_k\ge \widehat{X}$. From \eqref{eq2.13}, we have
     \begin{align}\label{eq2.14}
     & ( A_{k-1}^{\top} - X_k G_{k-1})(\widehat{X} - X_k) + ( \widehat{X} - X_k)( A_{k-1} - G_{k-1} X_k)  \nonumber\\
    %  &= (\widehat{X} - X_k) G_{k-1} ( \widehat{X} - X_k)- \widehat{X} G_{k-1} \widehat{X} + A_{k-1}^{\top}\widehat{X} + \widehat{X} A_{k-1}  %\nonumber \\
    % + (X_k G_{k-1} X_k - A_{k-1}^{\top} X_k - X_k A_{k-1}) \nonumber \\
    %  &= (\widehat{X} - X_k) G_{k-1} ( \widehat{X} - X_k)- \widehat{X} G_{k-1} \widehat{X} + A_{k-1}^{\top}\widehat{X} + \widehat{X}A_{k-1} + H_{k-1}  \nonumber \\
     =& \ (\widehat{X} - X_k) G_{k-1} ( \widehat{X} - X_k) + \begin{bmatrix} \widehat{X} & -I \end{bmatrix} \Omega(X_{k-1}) \begin{bmatrix}
        \widehat{X}\\
     -I  \end{bmatrix}.
\end{align}
Since $X_{k-1}\ge \widehat{X}$ and $\widehat{X}$ is a solution of SCARE \eqref{eq:Stoc-ARE-N}, it follows from Lemma \ref{lem:increasing_Omega} that 
\begin{align*}
    \begin{bmatrix} \widehat{X} & -I \end{bmatrix} \Omega(X_{k-1}) \begin{bmatrix}
        \widehat{X}\\
     -I  \end{bmatrix}\ge \begin{bmatrix} \widehat{X} & -I \end{bmatrix} \Omega(\widehat{X}) \begin{bmatrix}
        \widehat{X}\\
     -I  \end{bmatrix}=0. 
\end{align*}
Since $G_{k-1}\ge 0$ and $A_{k-1} - G_{k-1} X_k$ is stable, it follows from \eqref{eq2.14} and Lemma \ref{lem3.2} that $X_k\ge \widehat{X}$. Since $X_0\ge \widehat{X}\ge 0$, we obtain that $X_k\ge \widehat{X}$ for each $k$. 

Next,  we show that the sequence $\{X_k\}$ is monotonically nonincreasing.  First, we show that 
$X_0\ge X_1$. Since $X_1$ is the stabilizing solution of 
\begin{align*}
    A_0^{\top}X+XA_0-XG_0X+H_0=0,
\end{align*}
%from \eqref{eq2.15}, 
we have
\begin{align*}
    &(A_0-G_0X_0)^{\top}(X_1-X_0)+(X_1-X_0)(A_0-G_0X_0)  %\\
    % &=  (X_{1} - X_0) G_{0} ( X_{1} - X_0) + A_{0}^{\top}X_{1} + X_{1} A_{0} - X_{1} G_{0} X_{1}+ (X_0 G_{0} X_0 - A_{0}^{\top} X_0 - X_0 A_{0}) \label{eq:decreasing_X1}\\
    =(X_0-X_1)G_0(X_0-X_1)-\mathcal{R}(X_0)\ge 0.  
\end{align*}
Since $(A_0-G_0X_0)$ is stable, from Lemma \ref{lem3.2}, $X_0\ge X_1$.
Suppose that $X_{k-1}\ge X_{k}\ge 0$ and $X_{k+1}$ stabilizing solution of 
\begin{align*}
 \begin{bmatrix} X &  -I \end{bmatrix} \Omega(X_{k})\begin{bmatrix}
        X\\
     -I  \end{bmatrix}=0.
\end{align*}
{\color{red} Since $X_k$ is stabilizing solution satisfying \eqref{eq2.13}, we have $A_{k-1} - G_{k-1} X_k$ is stable. On the other hand,}  
\begin{align*}
    & ( A_{k-1}^{\top} - X_k G_{k-1})(X_{k+1} - X_k) + ( X_{k+1} - X_k)( A_{k-1} - G_{k-1} X_k)  \nonumber\\
    % =& (X_{k+1} - X_k) G_{k-1} ( X_{k+1} - X_k)- X_{k+1} G_{k-1} X_{k+1} + A_{k-1}^{\top}X_{k+1} + X_{k+1} A_{k-1}  \nonumber \\
    % &+ (X_k G_{k-1} X_k - A_{k-1}^{\top} X_k - X_k A_{k-1}) \nonumber \\
    % =& (X_{k+1} - X_k) G_{k-1} ( X_{k+1} - X_k)- X_{k+1} G_{k-1} X_{k+1} + A_{k-1}^{\top}X_{k+1} + X_{k+1} A_{k-1} + H_{k-1} \nonumber \\
     =&  (X_{k+1} - X_k) G_{k-1} ( X_{k+1} - X_k) + \begin{bmatrix} X_{k+1} & -I \end{bmatrix} \Omega(X_{k-1}) \begin{bmatrix}
        X_{k+1}\\
     -I  \end{bmatrix}.
\end{align*}
From Lemma \ref{lem:increasing_Omega} and using the fact that $X_{k-1}\ge X_{k}$, we have
\begin{align*}
    \begin{bmatrix} X_{k+1} & -I \end{bmatrix} \Omega(X_{k-1}) \begin{bmatrix}
        X_{k+1}\\
     -I  \end{bmatrix}\ge \begin{bmatrix} X_{k+1} & -I \end{bmatrix} \Omega(X_k) \begin{bmatrix}
        X_{k+1}\\
     -I  \end{bmatrix}=0.
\end{align*}
Hence, $X_{k}\ge X_{k+1}$ by Lemma \ref{lem3.2}. Since $X_0\ge X_1$, we find that $\{ X_k \}_{k=0}^{\infty}$ is a monotonically nonincreasing sequence by mathematical induction. Finally, we show that $\mathcal{R}(X_k)\le 0$ for each $k$. Since  $\{ X_k \}_{k=0}^{\infty}$ is monotonically nonincreasing, from Lemma \ref{lem:increasing_Omega}  we have
\begin{align*}
    \mathcal{R}(X_k)=\begin{bmatrix} X_{k} & -I \end{bmatrix} \Omega(X_k) \begin{bmatrix}
        X_{k}\\
     -I  \end{bmatrix}\le \begin{bmatrix} X_{k} & -I \end{bmatrix} \Omega(X_{k-1}) \begin{bmatrix}
        X_{k}\\
     -I  \end{bmatrix}=0.
\end{align*}

(ii). Since the sequence $\{ X_k \}_{k=0}^{\infty}$ is monotonically nonincreasing and bounded below by $\widehat{X}$,  we have $\lim_{k\rightarrow \infty}X_k=\widehat{X}_{+}$, where $\widehat{X}_{+}$ is a solution of SCARE \eqref{eq:Stoc-ARE-N}. (iii). Using the fact that $A_k-G_kX_{k+1}$ is stable for each $k$, by continuity $\sigma\left(A_c(\widehat{X}_{+})-G_c(\widehat{X}_{+})\widehat{X}_{+}\right)\subset \mathbb{C}_-\bigcup i\mathbb{R}$. 
\end{proof}

\begin{remark}\label{rem3.1}
    From Theorems~\ref{thm:monotonical_increasing} and \ref{thm:monotonical_decreasing}, we obtain that $\widehat{X}_-$ and $\widehat{X}_+$ are minimal and maximal PSD solutions of SCARE \eqref{eq:Stoc-ARE-N}, respectively. In addition, we obtain that all eiganvalues of $A_c(\widehat{X}_-)-G_c(\widehat{X}_-)\widehat{X}_-$ and $A_c(\widehat{X}_+)-G_c(\widehat{X}_+)\widehat{X}_+$ are in $\mathbb{C}_-\bigcup i\mathbb{R}$. If the PSD solution of SCARE \eqref{eq:Stoc-ARE-N} is unique, we have $\widehat{X}_-=\widehat{X}_+$.
\end{remark}

From Theorems~\ref{thm:monotonical_increasing} and \ref{thm:monotonical_decreasing}, we proved that the sequence $\{ X_k \}_{k = 1}^{\infty}$ generated by the FP-CARE\_SDA   converges monotonically.
Suppose that
\begin{align*}
      \lim_{k \to \infty} X_k = \widehat{X} \geq 0.
\end{align*}
From \eqref{eq:kth-SARE_CARE}, we have the implicit relation between $X_k$ and $X_{k+1}$ as
\begin{align*}
      \mathcal{R}(X_k, X_{k+1}) &= A_{k}^{\top} X_{k+1} + X_{k+1} A_{k} - X_{k+1} G_k X_{k+1} + H_k = 0,  
\end{align*}
where $R_k = R + \Pi_{22}(X_{k})$, $G_k = B R_k^{-1} B^{\top}$, $A_k=A-BR_k^{-1}(L+\Pi_{12}(X_k))^{\top}$ and $H_k = Q + \Pi_{11}(X_{k}) - (L+\Pi_{12}(X_{k})) R_k^{-1} (L + \Pi_{12}(X_{k}))^{\top}$.
Let $\mathcal{H}^n$ be the set of all Hermitian matrices and 
\begin{align*}
     \mathcal{R}_{X_k}^{\prime} : \mathcal{H}^n \to \mathcal{H}^n, \quad \mathcal{R}_{X_{k+1}}^{\prime} : \mathcal{H}^n \to \mathcal{H}^n
\end{align*}
be the Fr\`{e}chet partial derivative with respect to $X_k$ and $X_{k+1}$, respectively, i.e.,
\begin{subequations}
\begin{align}
      \mathcal{R}_{X_k}^{\prime}(Z) &= \lim_{t \to 0} \frac{1}{t}(\mathcal{R}(X_k+tZ) - \mathcal{R}(X_k)), \label{eq:Fd_k}\\
      \mathcal{R}_{X_{k+1}}^{\prime}(Z) &= \lim_{t \to 0} \frac{1}{t}(\mathcal{R}(X_{k+1}+tZ) - \mathcal{R}(X_{k+1})).  \label{eq:Fd_k+1}
\end{align}
\end{subequations}
From the implicit function theorem in Banach space, it holds that
\begin{align}
      \mathcal{R}_{X_k}^{\prime}(Z) + \mathcal{R}_{X_{k+1}}^{\prime}(Z) \frac{d X_{k+1}}{d X_k} = 0. \label{eq:IFT_RXk}
\end{align}
From \eqref{eq:IFT_RXk}, we can derive the derivation of the explicit expression of $X_{k+1} := F(X_k)$.

From \eqref{eq:Fd_k+1} follows that
\begin{align}
     \mathcal{R}_{X_{k+1}}^{\prime}(Z) = (A_k^{\top} - X_{k+1} G_k) Z + Z(A_k - G_k X_{k+1}). \label{eq:R_k+1_prime}
\end{align}
Furthermore, the Fr\`{e}chet derivatives of $H_c(X)$, $G_c(X)$ and $A_c(X)$ at $X=X_k$ are 
\begin{subequations} \label{eq:lim}
\begin{align}
 H_{c}^{'}(X_k)(Z) =&\Pi_{11}(Z)-\Pi_{12}(Z)R_c(X_k)^{-1}L_c(X_k)^{\top}-L_c(X_k)R_c(X_k)^{-1}\Pi_{12}(Z)^{\top}\nonumber\\
 &+L_c(X_k)R_c(X_k)^{-1}\Pi_{22}(Z)R_c(X_k)^{-1}L_c(X_k)^{\top}\nonumber\\
 =&\left[I, -L_c(X_k)R_c(X_k)^{-1}\right]\Pi(Z)\left[\begin{array}{c}
      I \\
      -R_c(X_k)^{-1}L_c(X_k)^{\top}
 \end{array}\right],\label{eq:lim_H}\\
  G_{c}^{'}(X_k)(Z) =&-BR_c(X_k)^{-1}\Pi_{22}(Z)R_c(X_k)^{-1}B^{\top},\label{eq:lim_G}\\
  A_{c}^{'}(X_k)(Z) =&BR_c(X_k)^{-1}\Pi_{22}(Z)R_c(X_k)^{-1}L_c(X_k)^{\top}-BR_c(X_k)^{-1}\Pi_{12}(Z)^{\top}.\label{eq:lim_A}
\end{align}
\end{subequations}
From \eqref{eq:Fd_k} and \eqref{eq:lim}, we have
\begin{align}\label{eq3.18}
    \mathcal{R}_{X_{k}}^{\prime}(Z) =\left(A_{c}^{'}(X_k)(Z)\right)^{\top}X_{k+1}+X_{k+1}A_{c}^{'}(X_k)(Z)-X_{k+1}G_{c}^{'}(X_k)(Z)X_{k+1}+H_{c}^{'}(X_k)(Z)
\end{align}

Denote 
\begin{align*}
    \widehat{A} = A_c(\widehat{X}),\ \
    \widehat{G}=G_c(\widehat{X}),\ \
    \widehat{L}=L_c(\widehat{X}),\ \
    \widehat{R} =R_c(\widehat{X}).
\end{align*} 
From \eqref{eq:R_k+1_prime}, \eqref{eq:lim} and \eqref{eq3.18}, we set
\begin{align*}
     \mathcal{L}_{\widehat{X}}(Z) \equiv& ( \widehat{A}^{\top} - \widehat{X} \widehat{G}) Z + Z ( \widehat{A} - \widehat{G} \widehat{X}),\\
      \Psi_{\widehat{X}}(Z) \equiv &\left(\widehat{L}\widehat{R}^{-1}\Pi_{22}(Z)\widehat{R}^{-1}B^{\top}-\Pi_{12}(Z)\widehat{R}^{-1}B^{\top}\right)\widehat{X}+\widehat{X}\left(B\widehat{R}^{-1}\Pi_{22}(Z)\widehat{R}^{-1}\widehat{L}^{\top}-B\widehat{R}^{-1}\Pi_{12}(Z)^{\top}\right)\\
      &+\widehat{X}B\widehat{R}^{-1}\Pi_{22}(Z)\widehat{R}^{-1}B^{\top}\widehat{X}+\left[I, -\widehat{L}\widehat{R}^{-1}\right]\Pi(Z)\left[\begin{array}{c}
      I \\
      -\widehat{R}^{-1}\widehat{L}^{\top}
 \end{array}\right].
\end{align*}

\begin{theorem} \label{thm:R-linear-convergence}
Suppose $\lim_{k \to \infty} X_k = \widehat{X} \geq 0$ monotonically by SDA-CARE. If $\rho(\mathcal{L}_{\widehat{X}}^{-1} \Psi_{\widehat{X}}) < 1$, then $\{ X_k \}_{k=1}^{\infty}$ converges to $\widehat{X}$ R-linearly with
\begin{align*}
      \limsup_{k \to \infty} \sqrt[k]{\| X_k - \widehat{X} \|} \leq \widehat{\gamma} < 1
\end{align*}
for some $0 < \widehat{\gamma} < 1$ and $\| \cdot \|$ being any matrix norm.
\end{theorem}
\begin{proof}
From implicit function theorem, there is a function $F$ such that
\begin{align}
     X_{k+1} = F(X_k), \quad \widehat{X} = F(\widehat{X}). \label{eq:fun_F}
\end{align}
Then
\begin{align*}
     X_{k+1} - \widehat{X} = F(X_{k}) - F(\widehat{X}) =  F^{\prime}(\widehat{X}) ( X_k - \widehat{X}) + o(\| X_k - \widehat{X} \|).
\end{align*}
By implicit function Theorem, we have
\begin{align}
     -F^{\prime}(\widehat{X})(Z) = \mathcal{L}_{\widehat{X}}^{-1} \Psi_{\widehat{X}}(Z). \label{eq:IFT_Fprime}
\end{align}
Since $\rho(\mathcal{L}_{\widehat{X}}^{-1} \Psi_{\widehat{X}}) < 1$, there exists a matrix norm $\| \cdot \hat{\|}$ and a constant $M$ such that 
$\| \cdot \| \leq M \| \cdot \hat{\|}$ and
\begin{align}
     \| \mathcal{L}_{\widehat{X}}^{-1} \Psi_{\widehat{X}} \hat{\|} < 1. \label{eq:nrm_lt_1}
\end{align}
Therefore, it implies
\begin{align*}
     \| X_{k+1} - \widehat{X} \| \leq \| F^{\prime}(\widehat{X})^k \| \| X_0 - \widehat{X} \| \leq M \| F^{\prime}(\widehat{X})^k \hat{\|} \| X_0 - \widehat{X} \|.
\end{align*}
That is
\begin{align*}
     \sqrt[k]{\| X_{k+1} - \widehat{X} \|} \leq M^{1/k} \| F^{\prime}(\widehat{X}) \hat{\|} \| X_0 - \widehat{X} \|^{1/k}.
\end{align*}
Hence, $\limsup_{k \to \infty} \sqrt[k]{\| X_k - \widehat{X} \|} \leq \widehat{\gamma} < 1$ for some $\widehat{\gamma} < 1$.
\end{proof}

\section{Newton  method and Modified Newton method} \label{sec:Newton_mNewton}
Newton's method for solving the SCARE \eqref{eq:Stoc-ARE-N} has been studied in \cite{dahi:2001}. The associated Newton's iteration is
\begin{align}
    X_{k+1} = X_k - (\mathcal{R}_{X_k}^{\prime})^{-1} \mathcal{R}(X_k), \ \mbox{ for } k = 0, 1, 2, \ldots, \label{eq:Newton_iteration}
\end{align}
provided that the Fr\`{e}chet  derivative $\mathcal{R}_{X_k}^{\prime}$ are all invertible. 
% where
% \begin{multline*} %\label{eq:Stoc-ARE-N}
% \mathcal{R}(X) \equiv A^{\top}X+XA+Q+\Pi_{11}(X) \\
%   -(XB+L+\Pi_{12}(X))(R+\Pi_{22}(X))^{-1}(XB+L+\Pi_{12}(X))^{\top}=0.
% \end{multline*}
The iteration in \eqref{eq:Newton_iteration} is equivalent to solve $X_{k+1}$ of the nonlinear matrix equation
\begin{subequations}
\begin{align}
    A_{X_k}^{\top} X + X A_{X_k} + \Pi_{X_k}(X) = - M_{X_k}, \label{eq:equ_Newton_iter}
\end{align}
where
\begin{align}
    A_{X_k} &= A - B R_{k}^{-1} S_{X_k}^{\top}= A_k - G_k X_k, \label{eq:NT-AXk} \\
    \Pi_{X_k}(X) &= \begin{bmatrix} 
        I & - S_{X_k} R_{k}^{-1}
    \end{bmatrix} \Pi(X) \begin{bmatrix}
        I \\ - R_{k}^{-1} S_{X_k}^{\top}
    \end{bmatrix}, \label{eq:NT-PiX} \\
    M_{X_k} &= \begin{bmatrix} 
        I & - S_{X_k} R_{k}^{-1}
    \end{bmatrix} \begin{bmatrix}
        Q & L \\ L^{\top} & R
    \end{bmatrix} \begin{bmatrix}
        I \\ - R_{k}^{-1} S_{X_k}^{\top}
    \end{bmatrix} \label{eq:NT-MXk}
\end{align}
with
\begin{align}
    S_{X_k} = X_k B+L_c(X_k) \equiv X_k B + L_k.
\end{align}
\end{subequations}
The sequence $\{ X_k \}_{k=0}^{\infty}$ computed by Newton's method converges to the maximal solution $\widehat{X}_{+}$ of \eqref{eq:Stoc-ARE-N} under mild conditions.

\begin{theorem}[\cite{dahi:2001}] \label{thm:convergence_Newton}
    Assume that there exists a solution $\widehat{X}$ to $\mathcal{R}(X) \geq0$ and a stabilizing matrix $X_0$ (i.e., $A_{X_0}$ is stable and $\rho((A_{X_0}^{\top} X_0 + X_0 A_{X_0})^{-1}\Pi_{X_0}) < 1$). Then,
    \begin{description}
        \item[(i)] $X_k \geq X_{k+1} \geq \widehat{X}$, $\mathcal{R}(X_k) \leq 0$ for all $k$. 
        \item[(ii)] $\lim_{k \to \infty} X_k = \widehat{X}_+$ is the maximal solution of \eqref{eq:Stoc-ARE-N}.
    \end{description}
\end{theorem}

The authors in \cite{dahi:2001} provided the convergence analysis in Theorem~\ref{thm:convergence_Newton}. However, it lacks the numerical methods for solving \eqref{eq:equ_Newton_iter}.
A method of solving \eqref{eq:equ_Newton_iter} is the following fixed-point iteration
\begin{align}
    A_{X_k}^{\top} Y_{j+1} + Y_{j+1} A_{X_k} = - (\Pi_{X_k}(Y_j) + M_{X_k})  \label{eq:Lyapunov} 
\end{align}
for $j = 0, 1, 2, \ldots$ with $Y_0 = X_k$.

The Lyapunov equation in \eqref{eq:Lyapunov} is a special case of CARE. As shown in \cite{hull:2018},  we can find the 
M\"{o}bius transformation $g(z) = \frac{z+\alpha}{z-\alpha}$ with $\alpha > 0$ to transform \eqref{eq:Lyapunov} into a DARE
\begin{subequations} \label{eq:DARE_lyap}
\begin{align}
    Y  = A_d^{\top} Y  A_d + H_d 
\end{align}
with
\begin{align}
    A_d &= I + 2 \alpha ( A_{X_k} - \alpha I)^{-1}, \\
    H_d &=  2 \alpha ( A_{X_k}^{\top} - \alpha I)^{-1} (\Pi_{X_k}(Y_j) + M_{X_k}) ( A_{X_k} - \alpha I)^{-1}.
\end{align}
\end{subequations}
Therefore, the Lyapunov equation \eqref{eq:Lyapunov} can be solved efficiently by L-SDA in Algorithm~\ref{alg:Lyapunov_SDA}. 
\begin{algorithm}
  \begin{algorithmic}[1]
    \REQUIRE Coefficient matrices $C \geq 0$ and $E$ (stable), parameter $\alpha>0$, the stopping tolerance $\varepsilon$.
    \ENSURE The solution $Y$.
    \STATE Compute $E_0 = I + 2 \alpha (E - \alpha I)^{-1}$ and $Y_0 = 2 \alpha (E^{\top} - \alpha I)^{-1} C ( E - \alpha I)^{-1}$.
    \STATE Set $k = 0$.
    \REPEAT
    \STATE Compute $E_{k+1} = E_k^2$ and $Y_{k+1} = Y_k + E_k^{\top} Y_k E_k$.
    \STATE Compute the normalized residual $\delta$ of Lyapunov equation.
    \STATE Set $k = k + 1$.
    \UNTIL{$\delta < \varepsilon$.} 
    \STATE Set $Y = Y_k$.
  \end{algorithmic}
  \caption{L-SDA for solving Lyapunov equation $E^{\top} Y + Y E + C = 0$}
  \label{alg:Lyapunov_SDA}
\end{algorithm}

\begin{algorithm}
\caption{Newton fixed-point Lyapunov (NT-FP-Lyap\_SDA) method for solving SCARE \eqref{eq:Stoc-ARE-N} } \label{alg:Newton fixed-point Lyapunov}
\begin{algorithmic}[1]
\REQUIRE $A, Q \in \mathbb{R}^{n \times n}$, $B, L  \in \mathbb{R}^{n \times m}$, $R = R^{\top} \in \mathbb{R}^{m \times m}$, $A_0^i \in \mathbb{R}^{n \times n}$, $B_0^i \in \mathbb{R}^{n \times m}$ for $i = 1, \ldots, r$, initial $X_0^{\top} = X_0$, and  a tolerance $\varepsilon$.
\ENSURE Solution $X$.
\STATE Set $k = 0$ and $\delta = \infty$.
%\STATE Choose an initial $X_0^{\top} = X_0$.
\WHILE{$\delta>\varepsilon$} \label{alg:outer_NT_while}
\STATE Compute $S_k = X_k B + L + \Pi_{12}(X_k)$,  $R_k = \Pi_{22}(X_k) + R$.
\STATE Compute $A_{X_k} = A - B R_k^{-1} S_k^{\top}$, $P_k = [I, -S_k R_k^{-\top}]^{\top}$ and $M_k =  P_k^{\top} \begin{bmatrix} 
    Q & L \\ L^{\top} & R
\end{bmatrix} P_k$.
\STATE Set $Y_{0} = X_k$, $j=0$ and $\delta_f = \infty$.
\WHILE{$\delta_f > \varepsilon$} \label{alg:inner_NT_FP_while}
\STATE Compute $\Pi_{Y_j} = P_k^{\top} \Pi(Y_{j}) P_k$ and $C_{j} = \Pi_{Y_{j}} + M_k$.
\STATE Use L-SDA to solve the stabilizing solution $Y_{{j+1}}$ of the Lyapunov equation 
\begin{align*}
      A_{X_k}^{\top} Y + Y A_{X_k}  + C_{j}  = 0.
\end{align*}
\STATE Compute the normalized residual $\delta_f$ of \eqref{eq:equ_Newton_iter} with $X = Y_{{j+1}}$ and set $j = j + 1$.
\ENDWHILE \label{alg:inner_NT_FP_end}
\STATE Set $X_{k+1} = Y_{j}$ and compute the normalized residual $\delta$ of \eqref{eq:Stoc-ARE-N} with $X = X_{k+1}$.
\STATE Set $k = k + 1$.
\ENDWHILE \label{alg:outer_NT_end}
\STATE Set $X = X_k$.
\end{algorithmic}
\end{algorithm}

Algorithm~\ref{alg:Newton fixed-point Lyapunov}, called the NT-FP-Lyap\_SDA algorithm,  states the processes of solving SCARE in \eqref{eq:Stoc-ARE-N} by combining Newton's iteration in \eqref{eq:equ_Newton_iter} with fixed-point iteration in \eqref{eq:Lyapunov}. 
Steps ~\ref{alg:outer_NT_while}-\ref{alg:outer_NT_end} are the outer iterations of Newton's method. The inner iterations in Steps~\ref{alg:inner_NT_FP_while}-\ref{alg:inner_NT_FP_end} are the fixed-point iteration to solve the nonlinear matrix equation in \eqref{eq:equ_Newton_iter}.

It is well-known that the fixed-point iteration in \eqref{eq:Lyapunov} has only a linear convergence if it converges. Even if the convergence of Newton's method is quadratic, the overall iterations for solving \eqref{eq:equ_Newton_iter} by \eqref{eq:Lyapunov} can be large. To reduce the iterations, as the proposed method in \cite{guo:2002a}, we modify the Newton's iteration in \eqref{eq:equ_Newton_iter} as
\begin{align}
    A_{X_k}^{\top} X + X A_{X_k} = -\Pi_{X_k}(X_k) - M_{X_k} \label{eq:m_Newton_iter}
\end{align}
and solve it by L-SDA in Algorithm~\ref{alg:Lyapunov_SDA}.  Solving \eqref{eq:m_Newton_iter} is equival to solve \eqref{eq:equ_Newton_iter} by using one iteration of fixed-point iteration in \eqref{eq:Lyapunov} with $Y_0 = X_k$. We state the modified Newton's iteration in Algorithm~\ref{alg:mNewton Lyapunov} and call it as mNT-FP-Lyap\_SDA algorithm.

\begin{algorithm}
\caption{mNT-FP-Lyap\_SDA  method for solving SCARE \eqref{eq:Stoc-ARE-N} } \label{alg:mNewton Lyapunov}
\begin{algorithmic}[1]
\REQUIRE $A, Q \in \mathbb{R}^{n \times n}$, $B, L  \in \mathbb{R}^{n \times m}$, $R = R^{\top} \in \mathbb{R}^{m \times m}$, $A_0^i \in \mathbb{R}^{n \times n}$, $B_0^i \in \mathbb{R}^{n \times m}$ for $i = 1, \ldots, r$, and  a tolerance $\varepsilon$.
\ENSURE Solution $X$.
\STATE Set $k = 0$ and $\delta = \infty$.
\STATE Choose an initial $X_0^{\top} = X_0$.
%
%\STATE
\WHILE{$\delta>\varepsilon$}
\STATE Compute $S_k = X_k B + L + \Pi_{12}(X_k)$,  $R_k = \Pi_{22}(X_k) + R$.
\STATE Compute $A_{X_k} = A - B R_k^{-1} S_k^{\top}$, $P_k = [I, -S_k R_k^{-\top}]^{\top}$ and $M_k =  P_k^{\top} \begin{bmatrix} 
    Q & L \\ L^{\top} & R
\end{bmatrix} P_k$.
%\STATE Set $Y_{0} = X_k$, $j=0$ and $\delta_f = \infty$.
%\WHILE{$\delta_f > \varepsilon$}
\STATE Compute $C_{k} = P_k^{\top} \Pi(X_{k}) P_k + M_k$.
\STATE Use L-SDA to solve the stabilizing solution $X_{{k+1}}$ of the Lyapunov equation 
\begin{align*}
      A_{X_k}^{\top} X + X A_{X_k}  + C_{k}  = 0.
\end{align*}
%\STATE Compute the relative residual $\delta_f$ of \eqref{eq:equ_Newton_iter} with $X = Y_{{j+1}}$ and set $j = j + 1$.
%\ENDWHILE
\STATE Compute the normalized residual $\delta$ of \eqref{eq:Stoc-ARE-N} with $X = X_{k+1}$.
\STATE Set $k = k + 1$.
\ENDWHILE 
\STATE Set $X = X_k$.
\end{algorithmic}
\end{algorithm}

In Theorem~\ref{thm:convergence_Newton}, a stabilizing matrix $X_0$ is needed to guarantee the convergence of Newton's method. However, the initial $X_0$ is generally not easy to find. To tackle this drawback, we integrate FP-CARE\_SDA (FPC) method in Algorithm~\ref{alg:SDA-CARE} into Algorithm~\ref{alg:Newton fixed-point Lyapunov}. We take a few iterations of Algorithm~\ref{alg:SDA-CARE}, e.g. $k = 5$, to obtain $X_k$ and then use $X_k$ as an initial $X_0$ in Algorithm~\ref{alg:Newton fixed-point Lyapunov}. We summarize it in Algorithm~\ref{alg:CARE Newton fixed-point Lyapunov}, called the FPC-NT-FP-Lyap\_SDA method.

\begin{algorithm}
\caption{FPC-NT-FP-Lyap\_SDA method for solving SCARE \eqref{eq:Stoc-ARE-N} } \label{alg:CARE Newton fixed-point Lyapunov}
\begin{algorithmic}[1]
\REQUIRE $A, Q \in \mathbb{R}^{n \times n}$, $B, L  \in \mathbb{R}^{n \times m}$, $R = R^{\top} \in \mathbb{R}^{m \times m}$, $A_0^i \in \mathbb{R}^{n \times n}$, $B_0^i \in \mathbb{R}^{n \times m}$ for $i = 1, \ldots, r$, and  a tolerance $\varepsilon$.
\ENSURE Solution $X$.
\STATE Set $k = 0$ and $\delta = \infty$.
\STATE Choose an initial $X_0^{\top} = X_0$.
\STATE Use FP-CARE\_SDA algorithm (Algorithm~\ref{alg:SDA-CARE}) with initial $X_0$ to compute an approximated $X_k$ satisfied $\| X_k - X_{k-1} \| < 0.01$.
\STATE Use NT-FP-Lyap\_SDA algorithm (Algorithm~\ref{alg:Newton fixed-point Lyapunov}) with initial matrix $X_k$ to solve SCARE \eqref{eq:Stoc-ARE-N}.
%
%\WHILE{$\delta>\varepsilon$}
%\STATE Compute $S_k = X_k B + L + \Pi_{12}(X_k)$,  $R_k = \Pi_{22}(X_k) + R$.
%\STATE Compute $A_{X_k} = A - B R_k^{-1} S_k^{\top}$, $P_k = [I, -S_k R_k^{-\top}]^{\top}$ and $M_k =  P_k^{\top} \begin{bmatrix} 
%    Q & L \\ L^{\top} & R
%\end{bmatrix} P_k$.
%\STATE Set $Y_{0} = X_k$, $j=0$ and $\delta_f = \infty$.
%\WHILE{$\delta_f > \varepsilon$}
%\STATE Compute $\Pi_{Y_j} = P_k^{\top} \Pi(Y_{j}) P_k$ and $C_{j} = \Pi_{Y_{j}} + M_k$.
%\STATE Use L-SDA to solve the stabilizing solution $Y_{{j+1}}$ of the Lyapunov equation \label{alg:step15-FPC-NT-FP-Lyap}
%\begin{align*}
%      A_{X_k}^{\top} Y + Y A_{X_k}  + C_{j}  = 0.
%\end{align*}
%\STATE Compute the normalized residual $\delta_f$ of \eqref{eq:equ_Newton_iter} with $X = Y_{{j+1}}$ and set $j = j + 1$.
%\ENDWHILE
%\STATE Set $X_{k+1} = Y_{j}$ and compute the normalized residual $\delta$ of \eqref{eq:Stoc-ARE-N} with $X = X_{k+1}$.
%\STATE Set $k = k + 1$.
%\ENDWHILE 
%\STATE Set $X = X_k$.
\end{algorithmic}
\end{algorithm}

\section{Convergence Analysis for Modified Newton Method} \label{sec:conv_mNewton}
In this section, we study the convergence of the modified Newton's method. For the special case that $L = 0$ and $B_0^i = 0$, $i = 1, \ldots, r$, Guo has shown the following result in \cite{guo:2002a}.
\begin{theorem}{{\rm \cite{guo:2002a}}}
    Let $\widehat{X} \ge 0$ be the solution of SCARE \eqref{eq:Stoc-ARE-N} with $L = 0$ and $B_0^i = 0$, $i = 1, \ldots, r$. Assume $X_0 \geq \widehat{X}$ such that $A_0 - G_0 X_0$ is stable and $\mathcal{R}(X_0) \leq 0$. Then $\{ X_k \}_{k=0}^{\infty}$ generated by mNT-FP-Lyap\_SDA algorithm is satisfied $\mathcal{R}(X_k) \leq 0$, $X_0 \geq X_1 \geq \cdots \geq X_k \geq \cdots \geq \widehat{X}$, and $\lim_{k \to \infty} X_k = \widehat{X}_+$, where $\widehat{X}_+$ is the maximal solution of SCARE \eqref{eq:Stoc-ARE-N}.
\end{theorem}

Now, we consider the general case, i.e., $B_0^i \neq 0$. Here, we assume that the sequence $\{ X_k \}_{k=0}^{\infty}$ exists and $X_k \geq 0$ for all $k$.
Under this posterior assumption, we have the monotonic convergence of $\{ X_k \}_{k=0}^{\infty}$.

\begin{lemma} \label{lem:X_to_stable}
     Let $C^{\top} C= Q - L R^{-1} L^{\top}$. Assume that % or 
     ${\rm Ker}C \subseteq \left(\bigcap_{i=1}^r{\rm Ker} A_0^i\right)\bigcap {\rm Ker}L$ and $(C, A)$ is detectable.  Suppose that $\{ X_k \}_{k=0}^{\infty}$ is a sequence generated by mNT-FP-Lyap\_SDA algorithm (Algorithm~\ref{alg:mNewton Lyapunov}). Then $A_k - G_k X_k$ is stable if $X_{k+1} \geq 0$.
\end{lemma}
\begin{proof}
%In the first, we prove that $A_k - G_k X_k$ is stable if $X_{k+1} \geq 0$.
%Let $X_{k+1} \geq 0$ be the solution of \eqref{eq:m_Newton_iter}. We prove that $A_k - G_k X_k$ is stable. 
From Proposition~\ref{prop:detectable}, $(H_k, A_k)$ is detectable. Let $H_k = C_k^{\top} C_k$.
From \eqref{eq:NT-AXk} and \eqref{eq:Pi_k+M_k}, we have
\begin{align*}
    (A_k - G_k X_k)^{\top} X_{k+1} + X_{k+1} (A_k - G_k X_k) = - (H_k + X_k G_k X_k).
\end{align*}
Assume that $A_k - G_k X_k$ is not stable. Then there exist $y \neq 0$ and $\mbox{Re}(\lambda) \geq 0$ such that $(A_k - G_k X_k) y = \lambda y$ which implies that
\begin{align*}
    0 \leq 2 \mbox{Re}(\lambda) y^{*} X_{k+1} y = - (y^{*} H_k y + y^{*}X_k G_k X_k y). 
\end{align*}
Since $H_k, G_k \geq 0$, we obtain the following $H_k y = 0$ and $G_k X_k y = 0$.
% \begin{align*}
%     H_k y = 0 \quad \mbox{ and } \quad G_k X_k y = 0.
% \end{align*}
This means that
\begin{align*}
    \lambda y = (A_k - G_kX_k) y = A_k y, \quad y \neq 0, \quad \mbox{Re}(\lambda) \geq 0, \quad \mbox{ and } \quad H_k y = 0,
\end{align*}
which contradicts the detectability of $(H_k, A_k)$.
\end{proof}

From \eqref{eq:mtx_LRQ_c}, \eqref{eq:NT-PiX} and \eqref{eq:NT-MXk}, we have
\begin{align}
     \Pi_{X_k}(X_k) + M_{X_k}  
    &= \begin{bmatrix} 
        I & - S_{X_k} R_c(X_k)^{-1}
    \end{bmatrix} \begin{bmatrix}
        Q_c(X_k) & L_c(X_k) \\ L_c(X_k)^{\top} & R_c(X_k) 
    \end{bmatrix} \begin{bmatrix}
        I \\ - R_c(X_k)^{-1} S_{X_k}^{\top}
    \end{bmatrix} \nonumber \\
    &= \begin{bmatrix} 
        I & - L_c(X_k) R_c(X_k)^{-1}
    \end{bmatrix} \begin{bmatrix}
        Q_c(X_k) & L_c(X_k) \\ L_c(X_k)^{\top} & R_c(X_k) 
    \end{bmatrix} \begin{bmatrix}
        I \\ - R_c(X_k)^{-1} L_c(X_k)^{\top}
    \end{bmatrix} \nonumber \\
    &+ \begin{bmatrix} 
        I & - L_c(X_k) R_c(X_k)^{-1}
    \end{bmatrix} \begin{bmatrix}
        Q_c(X_k) & L_c(X_k) \\ L_c(X_k)^{\top} & R_c(X_k) 
    \end{bmatrix} \begin{bmatrix}
        0 \\ - R_c(X_k)^{-1} B^{\top} X_k
    \end{bmatrix} \nonumber \\
    &+ \begin{bmatrix} 
        0 & - X_k B R_c(X_k)^{-1}
    \end{bmatrix} \begin{bmatrix}
        Q_c(X_k) & L_c(X_k) \\ L_c(X_k)^{\top} & R_c(X_k) 
    \end{bmatrix} \begin{bmatrix}
        I \\ - R_c(X_k)^{-1} S_{X_k}^{\top}
    \end{bmatrix} \nonumber \\
    &= H_c(X_k) + X_k B R_c(X_k)^{-1} B^{\top} X_k. \label{eq:Pi_k+M_k}
\end{align}
Substituting \eqref{eq:NT-AXk} and \eqref{eq:Pi_k+M_k} into \eqref{eq:m_Newton_iter}, the modified Newton's iteration in \eqref{eq:m_Newton_iter} is equivalent to solve the following equation
\begin{align}
    A_k^{\top} X + X A_k - X G_k X + H_k + (X - X_k) G_k (X - X_k) = 0. \label{eq:equiv_mNT}
\end{align}

The matrix $H_c(X)$ in \eqref{eq:mtx_Hc} is PSD and satisfies $H_c(X)\ge H_c(Y)$ for $X\ge Y\ge 0$. Suppose that $\widehat{X}_-\ge 0$ in Theorem \ref{thm:monotonical_increasing} is the minimal PSD solution of SCARE \eqref{eq:Stoc-ARE-N}. From \eqref{eq:lim_H}, the Fr\`{e}chet derivative $H_{c}^{'}(\widehat{X}_-):\mathcal{H}^n\rightarrow \mathcal{H}^n$ of $H_c(X)$ at $\widehat{X}_-\ge 0$ is given by
\begin{align}\label{eq5.3}
 H_{c}^{'}(\widehat{X}_-)(Z)%=&Q_{c}^{'}(\widehat{X}_-) (Z)-L_{c}^{'}(\widehat{X}_-) (Z)R_{c}^{-1}L_c^{\top}\nonumber\\&+L_cR_{c}^{-1} R_{c}^{'}(\widehat{X}_-)(Z)R_{c}^{-1}L_c^{\top}-L_cR_{c}^{-1}\left(L_{c}^{'}(\widehat{X}_-) (Z)\right)^{\top}\nonumber\\
 %=&\Pi_{11}(Z)-\Pi_{12}(Z)R_c^{-1}L_c^{\top}+L_cR_c^{-1}\Pi_{22}(Z)R_c^{-1}L_c^{\top}-L_cR_c^{-1}\Pi_{12}(Z)^{\top}\nonumber\\
 =&[I, -L_cR_c^{-1}]\Pi(Z)\left[\begin{array}{c}
      I \\
      -R_c^{-1}L_c^{\top}
 \end{array}\right],%\equiv\Pi_{\widehat{X}_-}(Z),
\end{align}
where $Z\in \mathcal{H}^n$ is Hermitian, $R_c=R_c(\widehat{X}_-)$, and $L_c=L_c(\widehat{X}_-)$. Note that if $Z\ge 0$, then $H_{c}^{'}(\widehat{X}_-)(Z)\ge 0$.
In the following theorem, we show the monotonically nondecreasing property if the initial $X_0$ is close enough to the solution $\widehat{X}_-$. 
\begin{theorem}\label{thm5.4}
    Let $\widehat{X}_-\ge 0$ in Theorem \ref{thm:monotonical_increasing} be the minimal PSD solution of SCARE \eqref{eq:Stoc-ARE-N}. Assume that there exists $\alpha >0$ such that  the Fr\`{e}chet derivative of $H_c(X)$ at $X=\widehat{X}_-$ satisfies $H_{c}^{'}(\widehat{X}_-)(Z)\ge \alpha Z$ for all $Z \ge0$. Let $X_0\le \widehat{X}_-$ be close enough to the solution $\widehat{X}_-$ and $\mathcal{R}(X_0)\ge 0$. Suppose that $\{X_k\}_{k=0}^{\infty}$ with $X_k\ge 0$ for all $k$ is a sequence generated by mNT-FP-Lyap\_SDA algorithm (Algorithm~\ref{alg:mNewton Lyapunov}). Then 
 \begin{description}
     \item[(i)] $X_k\le X_{k+1}$, $X_k\le \widehat{X}_-$, $\mathcal{R}(X_k)\ge 0$ for each $k\ge 0$.
     \item[(ii)] $\lim_{k\rightarrow \infty}X_k=\widehat{X}_-$.
 \end{description}
\end{theorem}
\begin{proof}
    Since $X_{k+1} \geq 0$, from Lemma~\ref{lem:X_to_stable}, $A_k - G_k X_k$ is stable for $k = 0, 1, \ldots$. We first show that 
    \begin{align}\label{eq5.4}
        \begin{bmatrix} Y&I\end{bmatrix}(\Omega(Y)-\Omega(X_{k}))\left[\begin{array}{c}
             Y  \\
            I 
        \end{array}\right]\ge (Y-X_k)G_k(Y-X_k),
    \end{align}
    for $X_0\le X_k\le Y\le \widehat{X}_-$. 
    From the first-order expansion, we have 
    \begin{align*}
        H_c(Y)-H_c(X_k)=H_{c}^{'}(X_k)(Y-X_k)+o(Y-X_k).
    \end{align*}
Using the fact that $H_{c}^{'}(\widehat{X}_-)(Z)\ge \alpha Z$ for $Z \ge 0$, $G_c(\widehat{X}_-)\ge 0$ is bounded and $X_0$ is sufficiently close to $\widehat{X}_-$, by continuity we have 
\begin{align*}
    H_c(Y)-H_c(X_k)\ge \frac{\alpha}{2}(Y-X_k) \text{ and }G_k=G_c(X_k)\text{ is bounded.}
\end{align*}
Hence,
\begin{align*}
    H_c(Y)-H_c(X_k)\ge (Y-X_k)G_k(Y-X_k).
\end{align*}
From Lemma \ref{lem:increasing_Omega}, we have $\Omega(Y)-\Omega(X_k)\ge 0$, and then
\begin{align*}
    \begin{bmatrix} Y&I\end{bmatrix}(\Omega(Y)-\Omega(X_{k}))\left[\begin{array}{c}
             Y  \\
            I 
        \end{array}\right]\ge \begin{bmatrix} Y&I\end{bmatrix}\left[\begin{array}{cc}
             0&0  \\
            0& H_c(Y)-H_c(X_k)
        \end{array}\right]\left[\begin{array}{c}
             Y  \\
            I 
        \end{array}\right]\ge(Y-X_k)G_k(Y-X_k).
\end{align*}
This shows that \eqref{eq5.4}.

Now, we prove assertion (i) by induction. For $k=0$, we already have $X_0\le \widehat{X}_-$ and $\mathcal{R}(X_0)\ge 0$. From \eqref{eq:Pi_k+M_k}, we have
\begin{align*}
      ( A_{0}^{\top} - X_{0} G_{0})(X_{1} - X_0) + ( X_{1} - X_0)( A_{0} - G_{0} X_{0}) =-\mathcal{R}(X_0)\le 0.
\end{align*} 
Since $A_0-G_0X_0$ is stable, by Lemma \ref{lem3.2}, we obtain that  $X_0\le X_1$. We now assume that assertion (i) is true for $k\ge 0$. From \eqref{eq:equiv_mNT}, we obtain that 
\begin{align*}
     & ( A_{k}^{\top} - X_{k} G_{k})(\widehat{X}_- - X_{k+1}) + ( \widehat{X}_- - X_{k+1})( A_{k} - G_{k} X_{k})  \nonumber\\
     &= \begin{bmatrix} \widehat{X}_- & -I \end{bmatrix} \Omega(X_{k}) \begin{bmatrix}
        \widehat{X}_-\\
     -I  \end{bmatrix}+(\widehat{X}_--X_{k})G_{k}(\widehat{X}_--X_{k}).
\end{align*} 
Using the fact that $\widehat{X}_-$ is a solution of SCARE, it follows from \eqref{eq5.4} that
\begin{align*}
    \begin{bmatrix} \widehat{X}_- & -I \end{bmatrix} \Omega(X_{k}) \begin{bmatrix}
        \widehat{X}_-\\
     -I  \end{bmatrix}+(\widehat{X}_--X_{k})G_{k}(\widehat{X}_--X_{k})\le 0.
\end{align*}
Therefore, $X_{k+1}\le \widehat{X}_-$ by Lemma \ref{lem3.2}.
Now, we show that $X_{k+1}\le X_{k+2}$ and $\mathcal{R}(X_{k+1})\ge 0$. Since 
\begin{align*}
    [X_{k+1},-I]\Omega(X_k)\begin{bmatrix}
        X_{k+1}\\
     -I  \end{bmatrix}+(X_{k+1} - X_k) G_k (X_{k+1} - X_k)= 0
\end{align*}
and $X_{k}\le X_{k+1}\le \widehat{X}_-$, we have
\begin{align*}
    \mathcal{R}(X_{k+1})&\equiv \begin{bmatrix} X_{k+1} & -I \end{bmatrix} \Omega(X_{k+1}) \begin{bmatrix}
        X_{k+1}\\
     -I  \end{bmatrix}\\&=
     [X_{k+1},-I](\Omega(X_{k+1})-\Omega(X_k))\begin{bmatrix}
        X_{k+1}\\
     -I  \end{bmatrix}- (X_{k+1} - X_k) G_k (X_{k+1} - X_k).
\end{align*}
From \eqref{eq5.4}, we have $\mathcal{R}(X_{k+1})\ge 0$. Then we obtain that 
\begin{align*}
     & ( A_{k+1}^{\top} - X_{k+1} G_{k+1})(X_{k+2} - X_{k+1}) + ( X_{k+2} - X_{k+1})( A_{k+1} - G_{k+1} X_{k+1})  \nonumber\\
     &=-\begin{bmatrix} X_{k+1} & -I \end{bmatrix} \Omega(X_{k+1}) \begin{bmatrix}
        X_{k+1}\\
     -I  \end{bmatrix}=-\mathcal{R}(X_{k+1})\le 0.
\end{align*}
Hence, $X_{k+1}\le X_{k+2}.$ The induction process is complete.

(ii). From the assertion (i), we have the sequence $\{X_k\}_{k=0}^{\infty}$ is monotonically nondecreasing and bounded above by $\widehat{X}_-$. This implies that $\lim_{k\rightarrow \infty}X_k=\widehat{X}_-$.   
\end{proof}

Next, we show the monotonically nonincreasing property and the convergence of $\{ X_k \}_{k=0}^{\infty}$ as follows.

\begin{theorem} \label{thm:monotonical_decreasing_mNT}
      Let $\widehat{X} \ge 0$ be the solution of SCARE \eqref{eq:Stoc-ARE-N}.  Let $X_0\ge \widehat{X}$ such that $A_0-G_0X_0$ is stable and $\mathcal{R}(X_0) \leq 0$.
Suppose that $\{ X_k \}_{k=0}^{\infty}$ with $X_k\geq 0$ for all $k$ is a sequence generated by mNT-FP-Lyap\_SDA algorithm (Algorithm~\ref{alg:mNewton Lyapunov}). % \textcolor{blue}{such that $A_{X_k}$ is stable}.
 Then 
 \begin{description}
     \item[(i)] $X_k\ge X_{k+1}$, $X_k\ge \widehat{X}$, $\mathcal{R}(X_k)\le 0$ for each $k\ge 0$.
     \item[(ii)] $\lim_{k\rightarrow \infty}X_k=\widehat{X}_+$, where $\widehat{X}_+\ge \widehat{X}$ is a solution of SCARE \eqref{eq:Stoc-ARE-N}.
     \item[(iii)] $\sigma\left(A_c(\widehat{X}_+)-G_c(\widehat{X}_+)\widehat{X}_+\right)\subset \mathbb{C}_-\bigcup i\mathbb{R}$.
 \end{description} 
\end{theorem}
\begin{proof}
(i). Since $X_{k+1} \geq 0$, from Lemma~\ref{lem:X_to_stable}, $A_k - G_k X_k$ is stable  for $k = 0, 1, \ldots$.
Now, we prove by induction that for each $k\ge 0$,
\begin{align}\label{eq5.2}
    X_k \ge X_{k+1},\ \ X_k\ge \widehat{X} \text{ and }\mathcal{R}(X_k)\le 0.
\end{align}
For $k=0$, we already have $X_0\ge \widehat{X}$ and $\mathcal{R}(X_0)\le 0$. 
%Since $X_1 \geq 0$, from Lemma~\ref{lem:X_to_stable}, $A_0 - G_0 X_0$ is stable.
From \eqref{eq:Pi_k+M_k}, we have
\begin{align*}
      ( A_{0}^{\top} - X_{0} G_{0})(X_{1} - X_0) + ( X_{1} - X_0)( A_{0} - G_{0} X_{0}) =-\mathcal{R}(X_0)\ge 0.
\end{align*} 
Since $A_0 - G_0 X_0$ is stable, by Lemma \ref{lem3.2}, we obtain that  $X_0\ge X_1$. We now assume that \eqref{eq5.2} is true for $k\ge 0$. From \eqref{eq:equiv_mNT} and Lemma \ref{lem:increasing_Omega}, we obtain that 
\begin{align*}
     & ( A_{k}^{\top} - X_{k} G_{k})(\widehat{X} - X_{k+1}) + ( \widehat{X} - X_{k+1})( A_{k} - G_{k} X_{k})  %\nonumber\\
     = \begin{bmatrix} \widehat{X} & -I \end{bmatrix} \Omega(X_{k}) \begin{bmatrix}
        \widehat{X}\\
     -I  \end{bmatrix}+(\widehat{X}-X_{k})G_{k}(\widehat{X}-X_{k})\\
     &\ge \begin{bmatrix} \widehat{X} & -I \end{bmatrix} \Omega(\widehat{X}) \begin{bmatrix}
        \widehat{X}\\
     -I  \end{bmatrix}+(\widehat{X}-X_{k})G_{k}(\widehat{X}-X_{k}) %\\
     =(\widehat{X}-X_{k})G_{k}(\widehat{X}-X_{k})\ge 0.
\end{align*} 
Therefore, $X_{k+1}\ge \widehat{X}$ by Lemma \ref{lem3.2}.
Now, we show that $X_{k+1}\ge X_{k+2}$ and $\mathcal{R}(X_{k+1})\le 0$.   From Lemma \ref{lem:increasing_Omega}, $X_k\ge X_{k+1}$ and $X_{k+1}$ satisfies \eqref{eq:equiv_mNT}, we have
\begin{align*}
\mathcal{R}(X_{k+1})&\equiv A_{k+1}^{\top} X_{k+1} + X_{k+1} A_{k+1} - X_{k+1} G_{k+1} X_{k+1} + H_{k+1} %\\
=\begin{bmatrix} X_{k+1} & -I \end{bmatrix} \Omega(X_{k+1}) \begin{bmatrix}
        X_{k+1}\\
     -I  \end{bmatrix}\\&\le\begin{bmatrix} X_{k+1} & -I \end{bmatrix} \Omega(X_{k}) \begin{bmatrix}
        X_{k+1}\\
     -I  \end{bmatrix} %\\
     = - (X_{k+1} - X_{k}) G_{k} (X_{k+1} - X_k)\le 0.    
\end{align*}
Then we obtain that 
\begin{align*}
     & ( A_{k+1}^{\top} - X_{k+1} G_{k+1})(X_{k+2} - X_{k+1}) + ( X_{k+2} - X_{k+1})( A_{k+1} - G_{k+1} X_{k+1})  \nonumber\\
     &= \begin{bmatrix} X_{k+2} & -I \end{bmatrix} \Omega(X_{k+1}) \begin{bmatrix}
        X_{k+2}\\
     -I  \end{bmatrix}-\begin{bmatrix} X_{k+1} & -I \end{bmatrix} \Omega(X_{k+1}) \begin{bmatrix}
        X_{k+1}\\
     -I  \end{bmatrix}+(X_{k+2}-X_{k+1})G_{k+1}(X_{k+2}-X_{k+1})\\
     &=-\begin{bmatrix} X_{k+1} & -I \end{bmatrix} \Omega(X_{k+1}) \begin{bmatrix}
        X_{k+1}\\
     -I  \end{bmatrix}=-\mathcal{R}(X_{k+1})\ge 0.
\end{align*}
Hence, $X_{k+1}\ge X_{k+2}$. The induction process is complete. %Thus, the sequence $\{X_k\}$ is monotonically nonincreasing and bounded below by $\widehat{X}$. 

(ii). From the assertion (i), we have the sequence $\{X_k\}_{k=0}^{\infty}$ is monotonically nonincreasing and bounded below by $\widehat{X}$. This implies that $\lim_{k\rightarrow \infty}X_k=\widehat{X}_+$, where $\widehat{X}_+$ is a solution of SCARE \eqref{eq:Stoc-ARE-N}. (iii). Since $A_k-G_kX_{k}$ is stable for each $k$,  by continuity $\sigma\left(A_c(\widehat{X}_+)-G_c(\widehat{X}_+)\widehat{X}_+\right)\subset \mathbb{C}_-\bigcup i\mathbb{R}$.
\end{proof}

In practice, solving Lyapunov equation by Algorithm~\ref{alg:Lyapunov_SDA} outperforms solving CARE by SDA algorithm.
However, it is difficult to find an initial matrix $X_0$ satisfying the conditions in Theorem~\ref{thm5.4} or Theorem~\ref{thm:monotonical_decreasing_mNT}.  
Moreover, comparing the Lyapunov equation in \eqref{eq:m_Newton_iter} with the CARE in \eqref{eq:kth-SARE_CARE}, the CARE is more approximated the SCARE \eqref{eq:Stoc-ARE-N} than the Lyapunov equation. 
The convergence of FP-CARE\_SDA algorithm will be better than that of mNT-FP-Lyap\_SDA algorithm as shown in Figure~\ref{fig:spectial_X0_convergence}.

To preserve the advantage and tackle the drawbacks of mNT-FP-Lyap\_SDA algorithm, we proposed a practical algorithm that integrates FP-CARE\_SDA and mNT-FP-Lyap\_SDA algorithms. 
If we ignore the quadratic term $(X - X_k) G_k (X - X_k)$ in \eqref{eq:equiv_mNT}, then \eqref{eq:equiv_mNT} is equal the CARE in \eqref{eq:kth-SARE_CARE}. 
Assume that $X_k \geq 0$ in \eqref{eq:equiv_mNT} is closed to the nonnegative solution of the SCARE \eqref{eq:Stoc-ARE-N}. 
The norm of this  quadratic term will be small, which means that FP-CARE\_SDA and mNT-FP-Lyap\_SDA algorithms will have the similar convergence. Furthermore, from Theorem~\ref{thm:monotonical_increasing}, $\{X_k\}_{k=0}^{\infty}$ produced by FP-CARE\_SDA with $X_0 = 0$ is monotonically nondecreasing and $\mathcal{R}(X_k) \geq 0$ for all $k\geq0$.
Therefore, we can use  FP-CARE\_SDA method in Algorithm~\ref{alg:SDA-CARE} with $X_0 = 0$ to obtain $X_k$, which is close enough to $\widehat{X}_{-}$ and $\mathcal{R}(X_k) \geq 0$. Based on Theorem~\ref{thm5.4}, such $X_k$ can be used as an initial matrix of the mNT-FP-Lyap\_SDA algorithm to improve the convergence of the mNT-FP-Lyap\_SDA algorithm. We call this integrating method as the FPC-mNT-FP-Lyap\_SDA method and state in Algorithm~\ref{alg:CARE mNewton Lyapunov}.

%Integrating FP-CARE\_SDA method in Algorithm~\ref{alg:SDA-CARE} to obtain an initial $X_0$, we develop a modified version of Algorithm~\ref{alg:CARE Newton fixed-point Lyapunov} in Algorithm~\ref{alg:CARE mNewton Lyapunov}, which calls the FPC-mNT-FP-Lyap\_SDA method.

\begin{algorithm}
\caption{FPC-mNT-FP-Lyap\_SDA  method for solving SCARE \eqref{eq:Stoc-ARE-N} } \label{alg:CARE mNewton Lyapunov}
\begin{algorithmic}[1]
\REQUIRE $A, Q \in \mathbb{R}^{n \times n}$, $B, L  \in \mathbb{R}^{n \times m}$, $R = R^{\top} \in \mathbb{R}^{m \times m}$, $A_0^i \in \mathbb{R}^{n \times n}$, $B_0^i \in \mathbb{R}^{n \times m}$ for $i = 1, \ldots, r$, and  a tolerance $\varepsilon$.
\ENSURE Solution $X$.
\STATE Set $k = 0$ and $\delta = \infty$.
\STATE Choose an initial $X_0^{\top} = X_0$.
\STATE Use FP-CARE\_SDA algorithm (Algorithm~\ref{alg:SDA-CARE}) with initial $X_0$ to compute an approximated $X_k$ satisfied $\| X_k - X_{k-1} \| < 0.01$.
\STATE Use mNT-FP-Lyap\_SDA algorithm (Algorithm~\ref{alg:mNewton Lyapunov}) with initial matrix $X_k$ to solve SCARE \eqref{eq:Stoc-ARE-N}.
\end{algorithmic}
\end{algorithm}

\section{Numerical experiments} \label{sec:numerical}
In this section, we construct eight examples of SCAREs, in which $A$, $B$, $Q$, $R$ come from the references and $A_0^1, \ldots, A_0^r$, $B_0^1, \ldots, B_0^r$ are constructed in this paper. 
In the first four examples, we find an initial $X_0$ satisfying the conditions in Theorems~\ref{thm:monotonical_decreasing} and \ref{thm:monotonical_decreasing_mNT} to numerically validate 
$\{ X_k \}_{k=0}^{\infty}$ produced by FP-CARE\_SDA and mNT-FP-Lyap\_SDA algorithms to be monotonically nonincreasing and compare the convergence of the FP-CARE\_SDA, NT-FP-Lyap\_SDA and mNT-FP-Lyap\_SDA algorithms. 
Such initial $X_0$ is difficult to find in Examples~\ref{example:6} - \ref{example:Quadrotor}, then  we use $X_0 = 0$ as the initial matrix. Theorem~\ref{thm:monotonical_increasing} shows $\{ X_k \}_{k=0}^{\infty}$ to be monotonically nondecreasing for the FP-CARE\_SDA algorithm. Numerical results show the failure of the convergence for the NT-FP-Lyap\_SDA and mNT-FP-Lyap\_SDA algorithms. However, the FPC-NT-FP-Lyap\_SDA and FPC-mNT-FP-Lyap\_SDA algorithms converge well to demonstrate the robustness of FP-CARE\_SDA algorithm.
%we randomly construct an initial $X_0$ with $\| X_0 \|_F \geq \| \widehat{X} \|_F$ to demonstrate the robustness of FP-CARE\_SDA algorithm. 

We use the normalized residual  
\begin{align*}
    \mbox{NRes}_k \equiv \frac{\| \mathcal{R}(X_k) \|_F}{2\| A X_k \|_F + \| Q \|_F + \| \Pi_{11}(X_k) \|_F + \| \mathcal{B}(X_k) \|_F} %\leq \varepsilon
\end{align*}
to measure the precision of the calculated solutions $X_k$,
where $\mathcal{B}(X) = (XB+L+\Pi_{12}(X))(R+\Pi_{22}(X))^{-1}(XB+L+\Pi_{12}(X))^{\top}$.
In Subsections~\ref{subsec:validation} and \ref{subsec:robustness}, we demonstrate the convergence behaviors of $\mbox{NRes}_k$ for FP-CARE\_SDA (Algorithm~\ref{alg:SDA-CARE}), NT-FP-Lyap\_SDA (Algorithm~\ref{alg:Newton fixed-point Lyapunov}), mNT-FP-Lyap\_SDA (Algorithm~\ref{alg:mNewton Lyapunov}), FPC-NT-FP-Lyap\_SDA (Algorithm~\ref{alg:CARE Newton fixed-point Lyapunov}), FPC-mNT-FP-Lyap\_SDA (Algorithm~\ref{alg:CARE mNewton Lyapunov}), and fixed-point iteration in \eqref{eq:FP-SCARE} proposed by Guo and Liang \cite{guli:2023}.  The efficiency of these algorithms is demonstrated in Subsection~\ref{subsec:efficiency}.

All computations in this section are performed in MATLAB 2023a with an Apple M2 Pro CPU.

\subsection{Numerical Validation of the convergence} \label{subsec:validation}
\begin{example} \label{example:1}
Let
\begin{align*}
     A &= \begin{bmatrix}
         0.9512 & 0 \\  0 & 0.9048
     \end{bmatrix}, \quad B = \begin{bmatrix}
         4.8770& 4.8770 \\ -1.1895 & 3.5690
     \end{bmatrix}, \quad Q = \begin{bmatrix}
         0.005 & 0 \\ 0 & 0.020
     \end{bmatrix}, \quad R = \begin{bmatrix}
         \frac{1}{3} & 0 \\  0 & 3
     \end{bmatrix}, \quad L = 0,\\
     A_0^1 &= \begin{bmatrix}
         -0.1 & 0.1 \\  -0.2 & 0.2
     \end{bmatrix}, \quad A_0^2 = \begin{bmatrix}
         1 &  -0.1 \\ 0.5 & 0
     \end{bmatrix}, \quad A_0^3 = \begin{bmatrix}
         0 & -0.2 \\ 0.2 & 0.5
     \end{bmatrix}, \\ 
     B_0^1 &= \begin{bmatrix}
         0 & -0.1 \\  0.1 & 0
     \end{bmatrix}, \quad B_0^2 = \begin{bmatrix}
         0.5 & 1 \\ -0.1 & 0.2
     \end{bmatrix}, \quad B_0^3 = \begin{bmatrix}
         1 & -1 \\ -0.2 & 1
     \end{bmatrix},
\end{align*}
where $A$, $B$, $Q$ and $R$ are given in Example 2.2 of \cite{abbe:1999b}.
\end{example}

\begin{example} \label{example:12}
Let
\begin{align*}
    A &= \varepsilon \begin{bmatrix}
         \frac{7}{3} & \frac{2}{3} & 0 \\
         \frac{2}{3} & 2 & - \frac{2}{3} \\
         0 & -\frac{2}{3} & \frac{5}{3}
    \end{bmatrix}, \quad B = \frac{1}{\sqrt{\varepsilon}} I_3, \quad A_0^1 =   0.1  \begin{bmatrix}
         0.1 & -0.1 & 0.01 \\  -0.2 & 0.1 & -0.1 \\  0.05 & -0.01 & 0.3
    \end{bmatrix},  \quad R = I_3, \quad L = 0,\\ 
    Q &= \begin{bmatrix}
         (4\varepsilon+4+\varepsilon^{-1})/9 &  2(2\varepsilon-1-\varepsilon^{-1})/9 & 2(2-\varepsilon-\varepsilon^{-1})/9 \\
         2(2\varepsilon-1-\varepsilon^{-1})/9 & (1+4\varepsilon+4/\varepsilon)/9  &    2(-1-\varepsilon+2/\varepsilon)/9 \\
                      2(2-\varepsilon-\varepsilon^{-1})/9 &  2(-1-\varepsilon+2/\varepsilon)/9  &   (4+\varepsilon+4/\varepsilon)/9
    \end{bmatrix}, \quad B_0^1 = 0.1 \begin{bmatrix}
         0 & 0 & 0.2 \\ 0.36 &  -0.6 & 0 \\  0 & -0.95 & -0.032
    \end{bmatrix}
\end{align*}
with $\varepsilon = 0.01$, where $A$, $B$, $Q$ and $R$ are given in Example 12 of \cite{chfl:2005}.
\end{example}

\begin{example} \label{example:Weng}
    Let
    \begin{align*}
        A &= \begin{bmatrix} 0.9512 & 0 \\  0 & 0.9048 \end{bmatrix}, \quad
        B = \begin{bmatrix} 4.8770 & 4.8770 \\ -1.1895& 3.5690 \end{bmatrix}, \quad Q = \begin{bmatrix} 0.0028 & -0.0013 \\ -0.0013 & 0.0190 \end{bmatrix}, \\ 
        R &= \begin{bmatrix} 1/3 & 0 \\ 0 & 3\end{bmatrix}, \quad
        A_0^1  = 6.5 \begin{bmatrix} 0.1 & 0.2 \\ 0.2 & 0.1 \end{bmatrix}, \quad B_0^1 = 6.5 I_2,\quad L = 0,
    \end{align*}
    modified in \cite{weph:2021}.
\end{example}

\begin{example} \label{example:rho_NT_GT_1}
    Let
    \begin{align*}
        A &= \begin{bmatrix} 3- \varepsilon & 1 \\ 4 & 2-\varepsilon \end{bmatrix}, \quad B = \begin{bmatrix} 1 \\ 1 \end{bmatrix}, \quad Q = \begin{bmatrix} 4\varepsilon-11 & 2\varepsilon-5 \\ 2\varepsilon-5 & 2\varepsilon-2 \end{bmatrix}, \quad A_0^1  = \begin{bmatrix} 0.1 & -0.1 \\ -0.2 & 0.1 \end{bmatrix}, \\
        B_0^1  &=  \begin{bmatrix} 0.1 \\ 0 \end{bmatrix}, \quad L  =  \begin{bmatrix} 0 \\ 0 \end{bmatrix}, \quad R = 1,\ \mbox{ with } \ \varepsilon = 5,
    \end{align*}
    where $A$, $B$, $Q$ and $R$ are given Example 11 of \cite{chfl:2005}.
\end{example}

\begin{figure}
\center
\begin{subfigure}[b]{0.42\textwidth}
\center
\includegraphics[width=\textwidth]{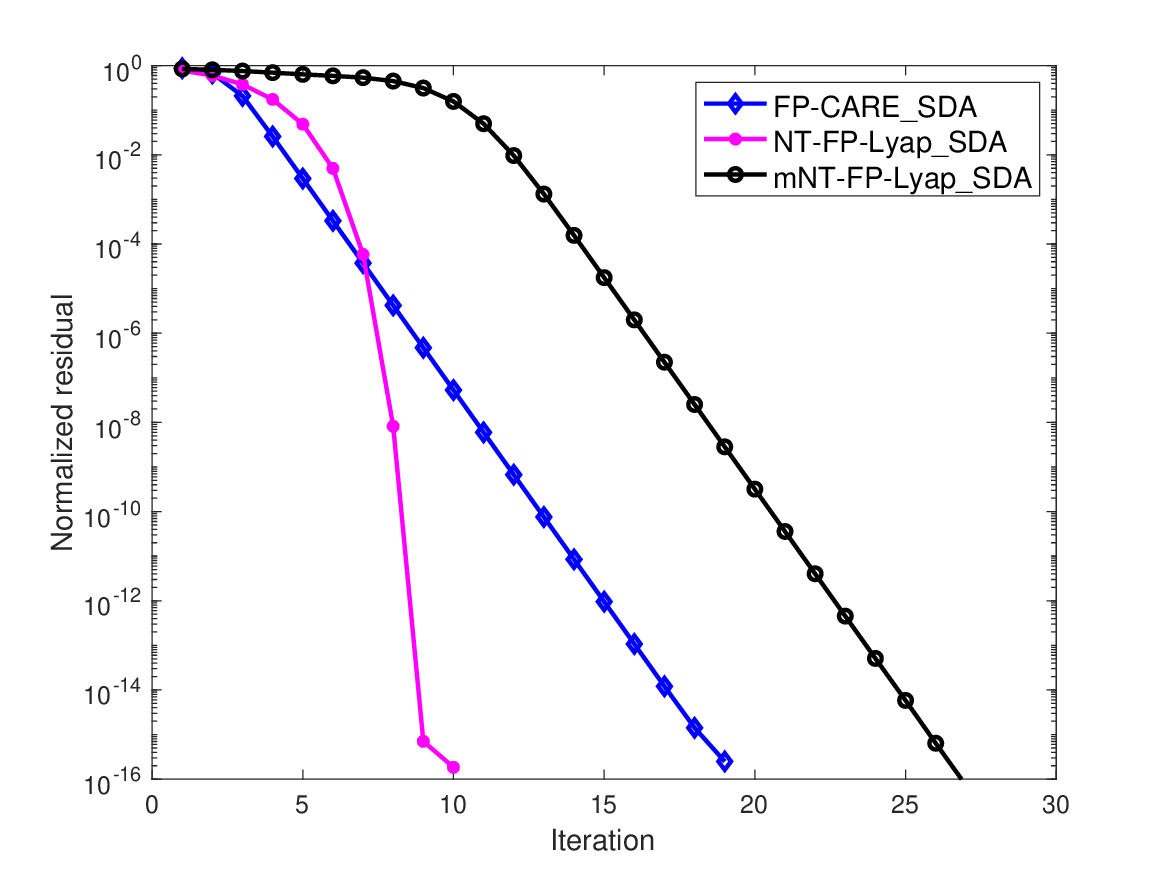}
\caption{Example~\ref{example:1}}
\label{fig:example1_convergence}
\end{subfigure}
\begin{subfigure}[b]{0.42\textwidth}
\center
\includegraphics[width=\textwidth]{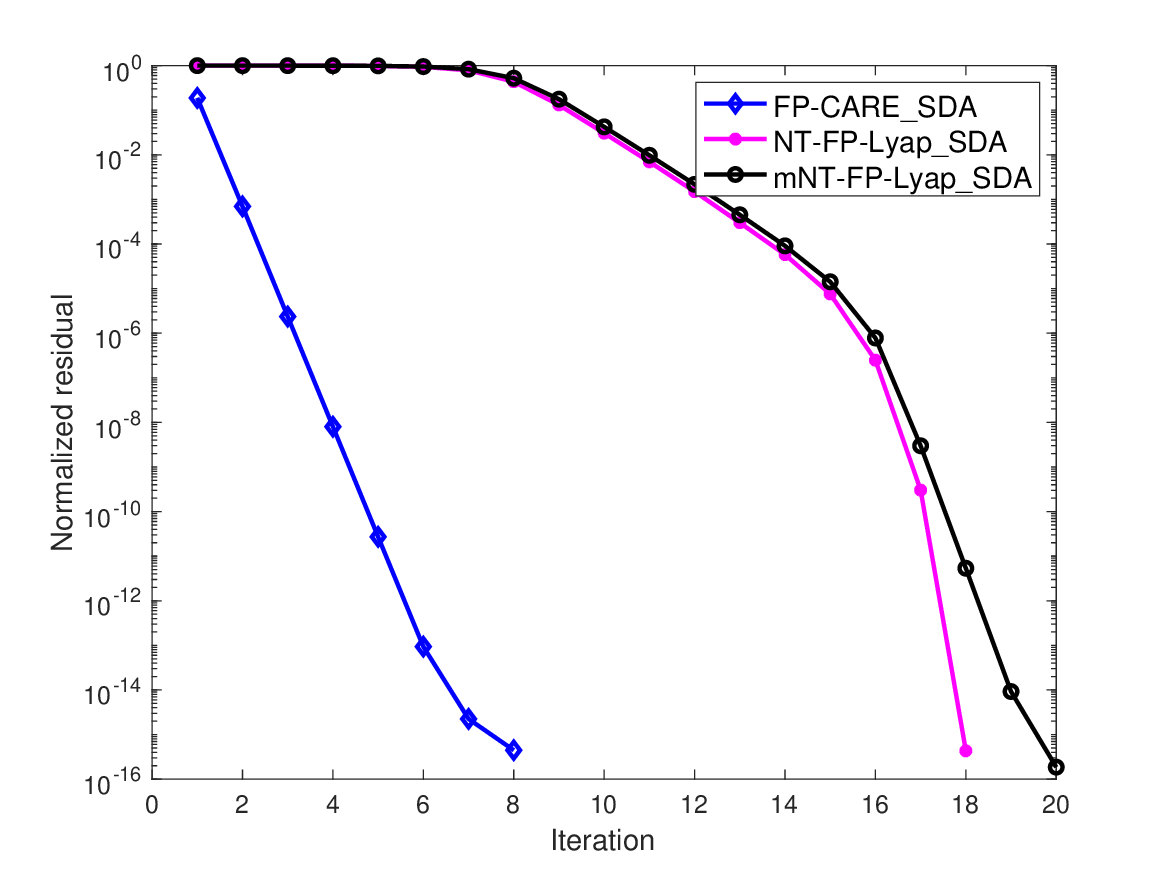}
\caption{Example~\ref{example:12} with $\varepsilon = 0.01$}
\label{fig:example12_convergence}
\end{subfigure}
\begin{subfigure}[b]{0.42\textwidth}
\center
\includegraphics[width=\textwidth]{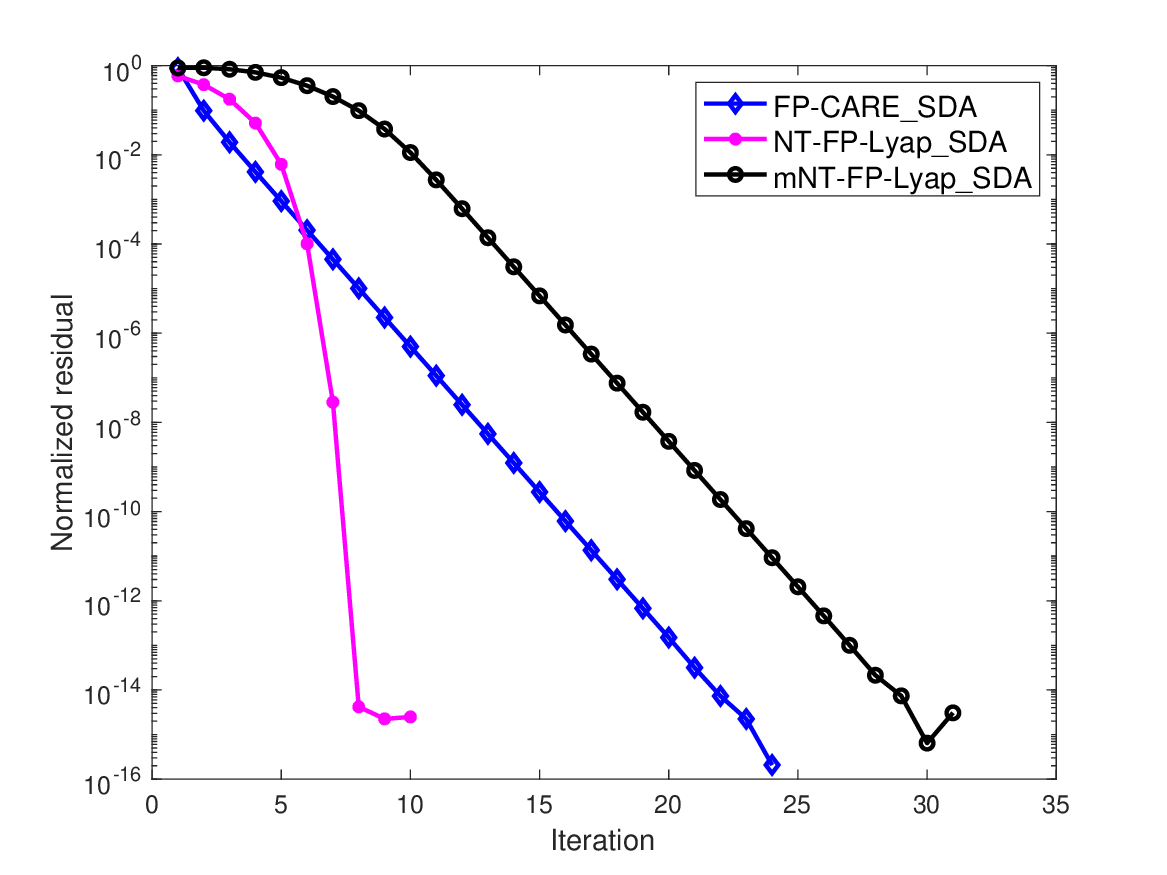}
\caption{Example~\ref{example:Weng}}
\label{fig:example_Weng}
\end{subfigure}
\begin{subfigure}[b]{0.42\textwidth}
\center
\includegraphics[width=\textwidth]{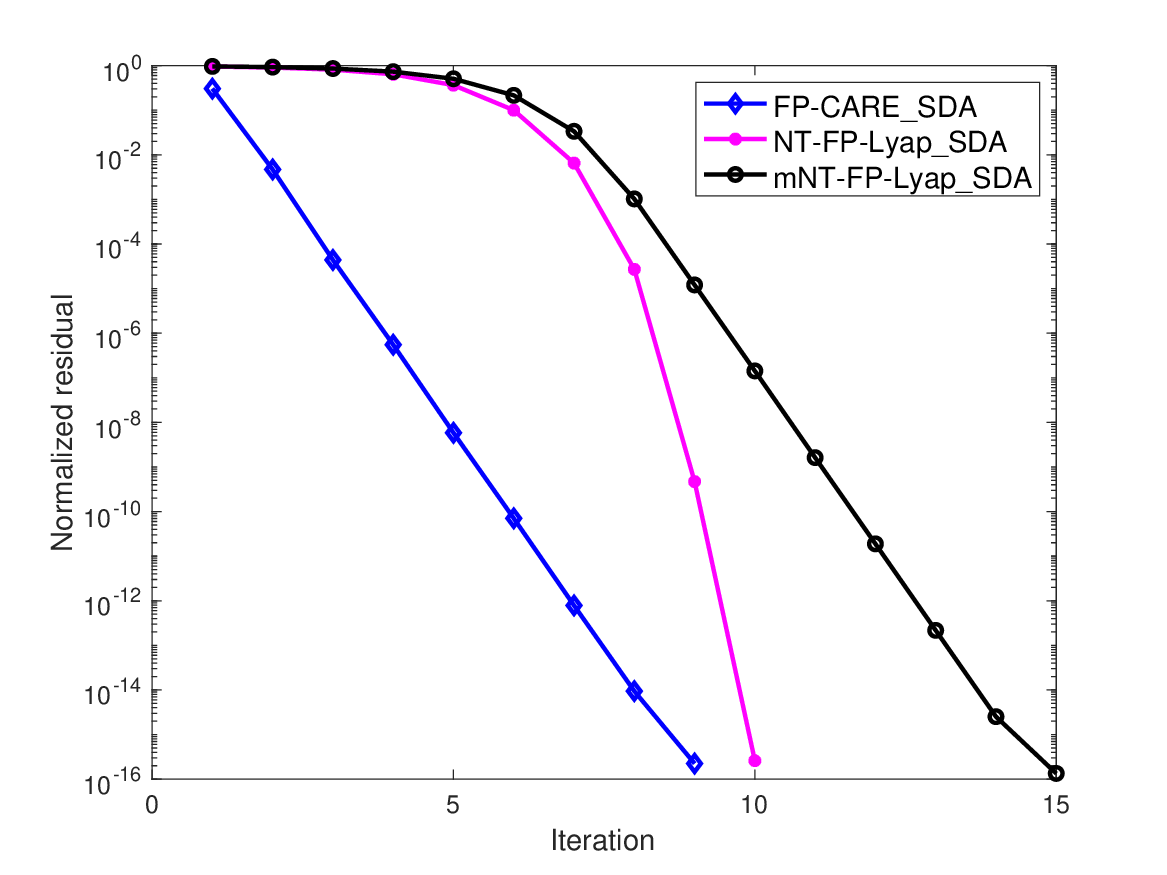}
\caption{Example~\ref{example:rho_NT_GT_1} with $\varepsilon = 5$}
\label{fig:rho_NT_GT_1}
\end{subfigure}
\caption{
 $\mbox{NRes}_k$ for FP-CARE\_SDA, FPC-NT-FP-Lyap\_SDA, FPC-mNT-FP-Lyap\_SDA and fixed-point method.}
\label{fig:spectial_X0_convergence}
\end{figure}

In Examples \ref{example:1} - \ref{example:rho_NT_GT_1}, we firstly take $X_0 = 0$ as the initial matrix for FP-CARE\_SDA, NT-FP-Lyap\_SDA and mNT-FP-Lyap\_SDA algorithms. FP-CARE\_SDA has a monotonically nondecreasing convergence as shown in Theorem~\ref{thm:monotonical_increasing}. The numerical result shows that (i) the produced sequence converges to the solution $\widehat{X}_{-} \geq 0$ of the SCARE; (ii) NT-FP-Lyap\_SDA and mNT-FP-Lyap\_SDA algorithms fail to converge to $\widehat{X}_{-}$. 

Next, we choose an initial $X_0$ with $\mathcal{R}(X_0) \leq 0$, $X_0 \geq \widehat{X}$, $A_0 - G_0 X_0$ being stable and $\rho(\mathcal{L}_{A_{X_0}}^{-1} \Pi_{X_0}) < 1$ so that $\{ X_k \}_{k=0}^{\infty}$ produced by the FP-CARE\_SDA, NT-FP-Lyap SDA and mNT-FP-Lyap SDA algorithms are monotonically nonincreasing and converge to the solution $\widehat{X}_{+}$ of the SCARE.  Using the same initial $X_0$ for FP-CARE\_SDA, NT-FP-Lyap\_SDA and mNT-FP-Lyap\_SDA algorithms, the normalized residual $\mbox{NRes}_k$  of the algorithms are shown in Figure~\ref{fig:spectial_X0_convergence}. Numerical results show that $\widehat{X}_{+} = \widehat{X}{-}$ for these four examples.

%\begin{itemize}
%\item For the initial $X_0$, $A_{X_0}$ is stable for these four examples,  $\rho(\mathcal{L}_{A_{X_0}}^{-1} \Pi_{X_0}) < 1$ for Examples~\ref{example:1} - \ref{example:Weng} and $\rho(\mathcal{L}_{A_{X_0}}^{-1} \Pi_{X_0}) > 1$ for Example~\ref{example:rho_NT_GT_1}. By Theorem~\ref{thm:convergence_Newton}, NT-FP-Lyap\_SDA method produces a non-increasing monotonic $\{ X_k \}_{k=0}^{\infty}$ for Examples~\ref{example:1} - \ref{example:Weng}. However, for Example~\ref{example:rho_NT_GT_1}, the initial matrix $X_0$ does not satisfy the convergent condition in Theorem~\ref{thm:convergence_Newton}. But, the numerical result shows that $\{ X_k \}_{k=0}^{\infty}$ is also non-increasing monotonic.   
%
%\item For the FP-CARE\_SDA algorithm, numerical results demonstrate that $\{ X_k \}_{k=0}^{\infty}$ is monotonically non-increasing and R-linearly convergent as shown in Theorems~\ref{thm:monotonical_decreasing} and \ref{thm:R-linear-convergence}.
%
%\item For mNT-FP-Lyap\_SDA algorithm, numerical results show that $\{ X_k \}_{k=0}^{\infty}$ keeps the monotonically non-increasing property and converges to the solution $\widehat{X}$.   
%\end{itemize}

The numerical results in Figure~\ref{fig:spectial_X0_convergence} show that the convergence of the FP-CARE\_SDA algorithm is obviously better than that of the mNT-FP-Lyap\_SDA algorithm. 
Furthermore, the normalized residuals at the first few iterations of the FP-CARE\_SDA algorithm are less than those of the NT-FP-Lyap\_SDA and mNT-FP-Lyap\_SDA algorithms. 
These indicate that we can use a few iterations of FP-CARE\_SDA algorithm to produce a good initial matrix $X_0$ for the NT-FP-Lyap\_SDA and mNT-FP-Lyap\_SDA algorithms, which are, called FPC-NT-FP-Lyap\_SDA and FPC-mNT-FP-Lyap\_SDA, respectively, used to improve the convergence and efficiency of the algorithms. 

\subsection{Robustness of FP-CARE\_SDA algorithm} \label{subsec:robustness}

Due to the difficulty in selecting the initial matrix for Theorems~\ref{thm:monotonical_decreasing} and \ref{thm:monotonical_decreasing_mNT}, based on Theorem~\ref{thm:monotonical_increasing}, we take $X_0 = 0$ as the initial matrix so that the FP-CARE\_SDA algorithm has monotonically nondecreasing convergence. Numerical results show that the NT-FP-Lyap\_SDA and mNT-FP-Lyap\_SDA algorithms with $X_0 = 0$ fail to converge. In the following, we introduce four typical practical examples and discuss the convergence of the FPC-NT-FP-Lyap\_SDA and FPC-mNT-FP-Lyap\_SDA algorithms for these examples to demonstrate the robustness of FP-CARE\_SDA algorithm.

\begin{example} \label{example:6}
The coefficient matrices, $A$, $B$, $Q$, $R$, and $L$ of a mathematical model of position and velocity controls for a string of high-speed vehicles \cite{abbe:1999a} can be described as, 
\begin{align*}
     A &= \begin{bmatrix}
          C & D \\ & \ddots & \ddots \\
          & & C & D \\
          & & & C & -1 \\
          & & & & -1
     \end{bmatrix} \in \mathbb{R}^{(2m-1) \times (2m-1)}, \quad C = \begin{bmatrix}
         -1 & 0 \\ 1 & 0
     \end{bmatrix}, \quad D = \begin{bmatrix}
         0 & 0 \\ -1 & 0
     \end{bmatrix}, \\
     B &= [b_{ij}] \in \mathbb{R}^{(2m-1) \times m}, \quad b_{ij} = \begin{cases} 
        1,  & \mbox{ for } i = 1, 3, 5, \ldots, 2m-1, j = (i+1)/2,\\
        0, & \mbox{ otherwise},
        \end{cases} \\
     Q &= [q_{ij} ]\in \mathbb{R}^{(2m-1) \times (2m-1)}, \quad q_{ij} = \begin{cases}
     10, & \mbox{ for } i = 2, 4, 6, \ldots, 2m-1, \ j = i, \\
     0, & \mbox{ otherwise},
     \end{cases} \\
     R &= I_m, \quad L = 0_{(2m-1)\times m},  \\ %\quad A_1 = 0.01\| A \|_{\infty} \widehat{A}_1, \quad B_1 = 0.05\| B \|_{\infty} \widehat{B}_1
\end{align*}
where $m$ is the number of vehicles.
Furthermore, the multiplicative white noises are 
\begin{align*}
    A_0^i &= 0.1 \times i \times \frac{\| A \|_{\infty}}{\| \widehat{A}_0^i \|_{\infty}} \widehat{A}_0^i \text{ with }\widehat{A}_0^i = \mbox{\texttt{wgn}}(2m-1,2m-1,8\times i),   \mbox{ for } i = 1, \ldots, 5, \\
     B_0^i &= 0.15 \times i \times \frac{\| B \|_{\infty}}{\| \widehat{B}_0^i \|_{\infty}} \widehat{B}_0^i\text{ with }\widehat{B}_0^i = \mbox{\texttt{wgn}}(2m-1,m,3\times i),   \mbox{ for } i = 1, \ldots, 5, 
\end{align*}
where \texttt{wgn} is a MATLAB built-in function used to generate white Gaussian noise.  %Here, $A$, $B$, $Q$, $R$ and $L$ are in \cite{abbe:1999a}.
\end{example}

\begin{example} \label{example:3}
The differential SDRE with impact angle guidance strategies models the 3D missile/target interception engagement  \cite{lwhxwl:2023,nakm:2021a}. The nonlinear dynamics can be expressed as in \eqref{eq:dyna_SSDC}, where the system, the control and the white noise matrix are given by 
%\begin{align*}
%    \dot{x}&= A(x) x+B(x)u
%\end{align*}
\begin{subequations} \label{eq6.1}
\begin{align}
    A(x)&=
    \begin{bmatrix}
    0 & 1 &0 & 0 & 0
    \\ 0 &\frac{-2\dot{r}}{r} & 0 & \frac{-1}{2}\dot{\psi}\sin(2x_{1}+2\theta_{f}) & \frac{g_{1}}{z_{a}}
    \\ 0 & 0 & 0 & 1 & 0
    \\0 &2\dot{\psi}\tan(x_{1}+\theta_{f}) &0 &\frac{-2\dot{r}}{r} &\frac{g_{2}}{z_{a}}
    \\0 &0 &0 &0 &-\eta
    \end{bmatrix}\ \
    B(x)=
    \begin{bmatrix}
    0&0\\
    \frac{-\cos\theta_{M}}{r} &0 \\
    0&0\\
    \frac{\sin\theta_{M}\sin\psi_{M}}{r\cos\theta} &\frac{-\cos\psi_{M}}{r\cos\theta}\\
        0 &0 
    \end{bmatrix},\nonumber\\
     A_0^i &= 0.2 \times i \times  \frac{\| A \|_{\infty}}{\| \widehat{A}_0^i \|_{\infty}} \widehat{A}_0^i \ \text{ with }\ \widehat{A}_0^i = \mbox{\texttt{wgn}}(5,5,8\times i), \mbox{ for } i = 1, \ldots, 4, \label{eq6.1a}\\
     B_0^i &= 0.1 \times i \times  \frac{\| B \|_{\infty}}{\| \widehat{B}_0^i \|_{\infty}} \widehat{B}_0^i \ \text{ with } \ \widehat{B}_0^i = \mbox{\texttt{wgn}}(5,2,3\times i), \mbox{ for } i = 1, \ldots, 4, \label{eq6.1b}
\end{align}
\end{subequations}
in which $r$ measures the distance between missile and target, $\{ \psi, \theta\}$ (respective to $\{ \psi_{M}, \theta_M\}$) are the azimuth and elevation angles (respective to missile) corresponding to the initial frame (respective the line-of-sight (LOS) frame). 
Furthermore, the state vector $x = [x_1, x_2, x_3, x_4, x_5]^{\top}$ is defined by $x_1 = \theta - \theta_f$, $x_2 = \dot{x}_1$, $x_3 = \psi - \psi_f$, $x_4 = \dot{x}_3$ with $\theta_f$ and $\psi_f$ being prescribed final angles, $x_5$ is a slow varying stable auxiliary variable governed by $\dot{x}_5 = -\eta x_5$, $\eta > 0$, $g_1$ and $g_2$ are highly nonlinear functions on $r$, $\theta$, $\psi$, $\theta_f$, $\psi_f$, $\theta_T$, $\psi_T$, $a_T^z$ and $a_T^y$ (see \cite{nakm:2021a} for details), where $\{\psi_T, \theta_T\}$ are the azimuth and elevation angles of the target to the LOS frame, $a_T^z$ and $a_T^y$ are the lateral accelerations for the target, and $(u_1, u_2) = (a_M^z, a_M^y)$ is the control vector for the maneuverability of missile. 

For a fixed state $x$ at the state-dependent technique of SDRE \cite{lwhxwl:2023}, the coefficient matrices of the SCARE \eqref{eq:SCARE_gen} associated with the optimization problem \eqref{eq:cost_fun} under the nonlinear dynamics for the missile/target engagement have the forms.
 
\begin{align*}
     A &= \begin{bmatrix}
          0 & 1          & 0 & 0           & 0 \\  
          0 & 0.0696 & 0 & -0.0307 & -1.91 \times 10^{-4} \\ 
          0 & 0          & 0 & 1           & 0 \\ 
          0 & 0.123   & 0 & 0.0696 & 6.13 \times 10^{-4} \\
          0 & 0         & 0 & 0           & -0.1
     \end{bmatrix}, \quad B = \begin{bmatrix}
          0 & 0 \\ -9.13 \times 10^{-5} & 0 \\ 0 & 0 \\ 2.42 \times 10^{-5} & -1.30 \times 10^{-4} \\ 0 & 0
     \end{bmatrix}, \\
     Q &= {\rm diag}(1000, 1000, 1000, 1000, 0),  \quad R = I_2, \quad L = 0, \quad A_0^i \text{ and }B_0^i\text{ are as in \eqref{eq6.1}.}
\end{align*}
\end{example}

%\begin{example} \label{example:7}
%    Let $A$, $B$, $Q$, $R$ and $L$ be defined as in Example~\ref{example:6}. Let $A_1 = 0.05 \| A \|_{\infty} \texttt{randn}(2m-1)$ and $B_1 = 0$, where \texttt{randn} is a MATLAB built-in function that is used to generate a matrix with normally distributed pseudorandom numbers.
%\end{example}

\begin{figure}
\center
%\begin{subfigure}[b]{0.32\textwidth}
%\center
%\includegraphics[width=\textwidth]{figures/ex2_Newton_CARE_fp.eps}
%\caption{Example~\ref{example:2}}
%\label{fig:example2_convergence}
%\end{subfigure}

\begin{subfigure}[b]{0.42\textwidth}
\center
\includegraphics[width=\textwidth]{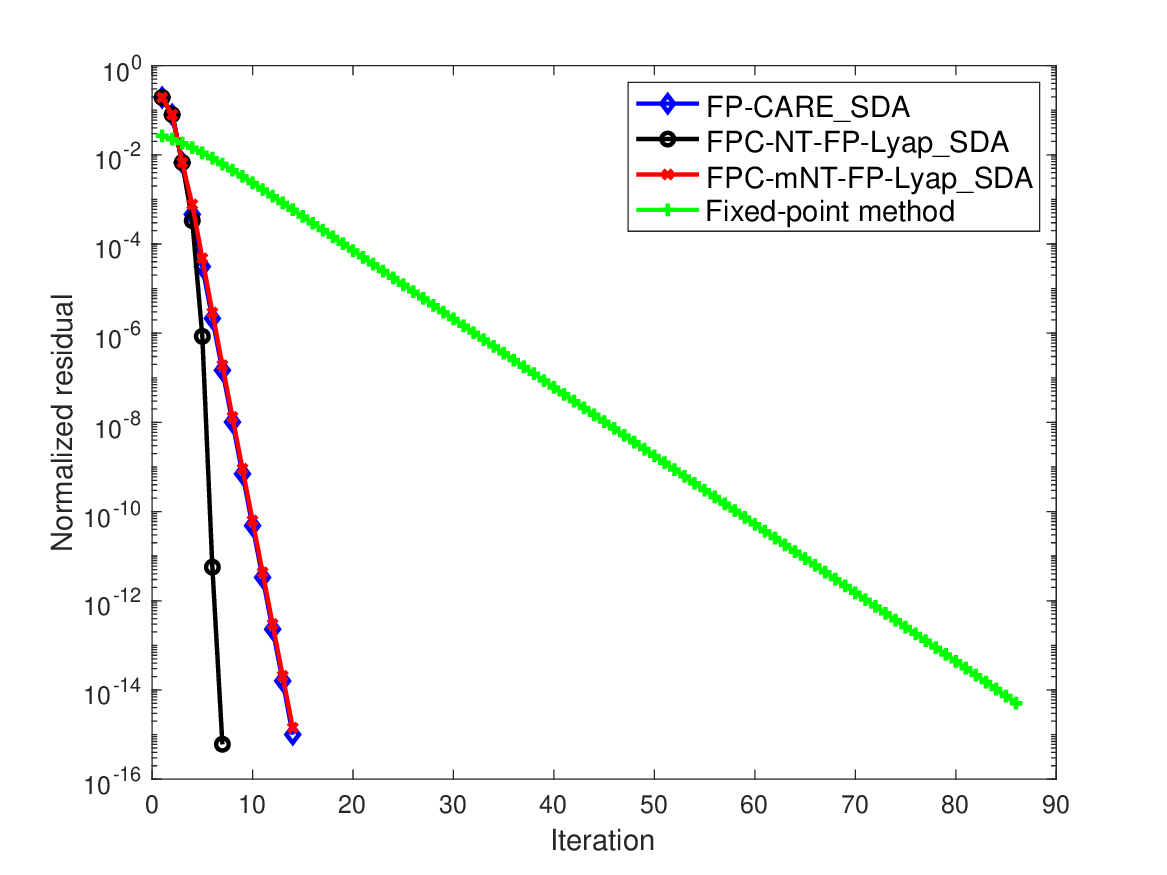}
\caption{Example~\ref{example:6} with $m = 100$}
\label{fig:example6_convergence}
%\includegraphics[width=\textwidth]{figures/ex7_N100.eps}
%\caption{Example~\ref{example:7} with $m = 100$}
%\label{fig:example7_convergence}
\end{subfigure}
\begin{subfigure}[b]{0.42\textwidth}
\center
\includegraphics[width=\textwidth]{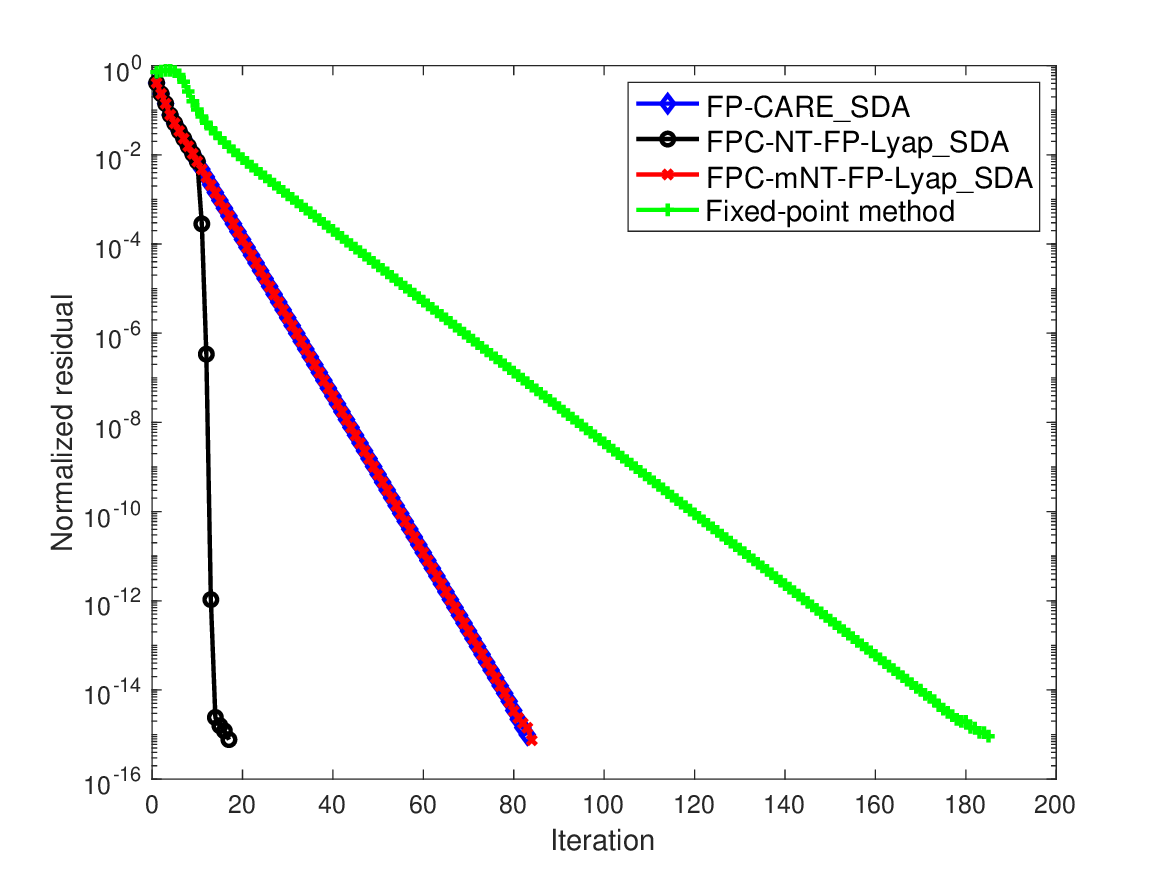}
\caption{Example~\ref{example:3}}
\label{fig:example3_convergence}
\end{subfigure}
\begin{subfigure}[b]{0.42\textwidth}
\center
\includegraphics[width=\textwidth]{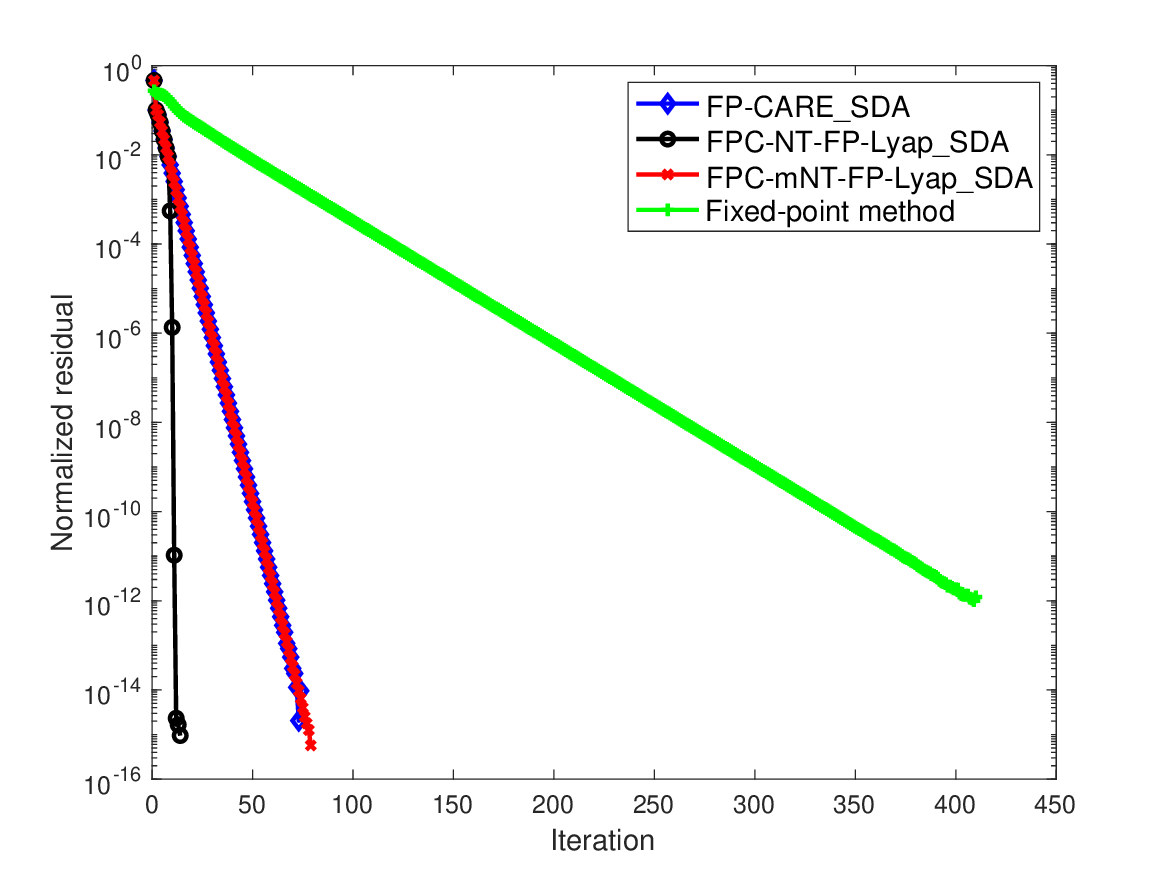}
\caption{Example~\ref{example:F16}}
\label{fig:exampleF16_convergence}
\end{subfigure}
\begin{subfigure}[b]{0.42\textwidth}
\center
\includegraphics[width=\textwidth]{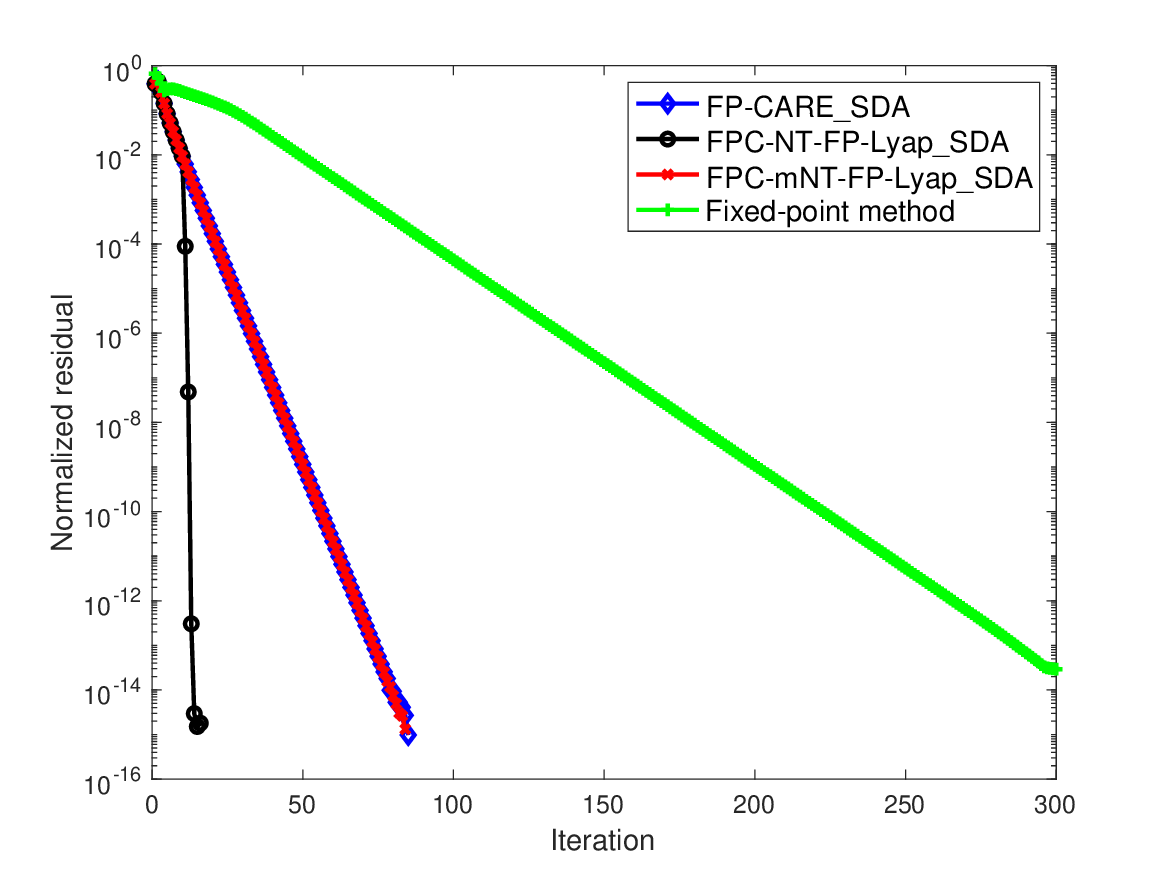}
\caption{Example~\ref{example:Quadrotor}}
\label{fig:exampleQuadrotor_convergence}
\end{subfigure}
%
%\begin{subfigure}[b]{0.32\textwidth}
%\center
%\includegraphics[width=\textwidth]{figures/ex8_N_20.eps}
%\caption{Example~\ref{example:9} with $m = 20$}
%\label{fig:example8_convergence}
%\end{subfigure}
\caption{
Normalized residual $\mbox{NRes}_k$ for FP-CARE\_SDA, FPC-NT-FP-Lyap\_SDA, FPC-mNT-FP-Lyap\_SDA and fixed-point method.}
\label{fig:1_k2_convergence}
\end{figure}

\begin{example}{\cite{cpws:2022}} \label{example:F16}
Finite-time SDRE of F16 aircraft flight control system \cite{cpws:2022} can be described as
\begin{align*}
      \dot{\mathbf{x}}\equiv \begin{bmatrix}\dot{u}\\
      \dot{v}\\
      \dot{w}\\
      \dot{p}\\
      \dot{q}\\
      \dot{r}\end{bmatrix} = & \begin{bmatrix}
           (g \sin\theta)/u & 0 & 0 & 0 & -w & v \\
           (-g\sin \phi \cos \theta)/u & 0 & 0 & w & 0 & -u \\
           (-g\cos \phi \cos \theta)/u & 0 & 0 & -v & u & 0 \\
           0 & 0 & 0 & c_1q/2 & (c_1p+c_2r)/2 & c_2 q/2 \\
           0 & 0 & 0 & c_3 p + c_4r/2 & 0 & c_4p/2 - c_3 r \\
           0 & 0 & 0 & c_5q/2 & (c_5p+c_6r)/2 & c_6 q/2 
      \end{bmatrix}\begin{bmatrix}u\\
      v\\
      w\\
      p\\
      q\\
      r\end{bmatrix}\\ &+\begin{bmatrix}
          1/m & 0 & 0 & 0 \\
          0 & 0 & 0 & 0 \\
          0 & 0 & 0 & 0 \\
          0 & c_{lp} & 0 & c_{np} \\
          c_{mq} Z_{TP} & 0 & c_{mq} & 0 \\
          0 & c_{lr} & 0 & c_{nr}
      \end{bmatrix} \begin{bmatrix}F_T\\
      L\\
      M\\
      N\end{bmatrix}
      \equiv A(\mathbf{x}) \mathbf{x} + B(\mathbf{x}) \mathbf{u},
\end{align*}
where $g$ is the gravity force, $m$ is the aircraft
mass, $(u, v, w)$, $(p, q, r)$ are the aircraft velocity and angular velocity vectors, respectively, for roll $\phi$, pitch $\theta$ and yaw $\psi$ angles, $c$'s parameters are the suitable combinations of coefficients of the aircraft inertial matrix, the control vector $\mathbf{u}=[F_T,L,M,N]^{\top}$ consists of the thrust, the resulted rolling, pitch and yawing moments from the layout of ailerons, elevators and rudders, $Z_{TP}$ is the position of the thrust point.

For a fixed state $\mathbf{x}$, the coefficient matrices of the SCARE \eqref{eq:SCARE_gen} with the optimization problem \eqref{eq:F_X*} has the forms
    \begin{align*}
        A &= \begin{bmatrix}
            3.958 \times 10^{-5}  &  0 &  0 & 0 & -5.866 & -6.985 \\
            2.116\times 10^{-4}  & 0 & 0 & 5.866  &  0  &  -84.66 \\
            -0.1158  & 0  & 0 & 6.985  & 84.66  &  0 \\
            0 & 0 & 0 & 1.791\times 10^{-4} & 4.303 \times 10^{-3}  &  -5.006 \times 10^{-3} \\
            0 & 0 & 0 & -5.329 \times 10^{-3} &  0  &  -4.259 \times 10^{-2} \\
            0 & 0 & 0 & -4.769 \times 10^{-3} &  3.253 \times 10^{-2} & -1.791 \times 10^{-4}
        \end{bmatrix},  \\ %\quad L = 0, \\
        B &= \begin{bmatrix}
            1.076 \times 10^{-4} & 0 & 0 & 0 \\
            0 & 0 & 0 & 0 \\ 
            0 & 0 & 0 & 0 \\
            0 & 7.780 \times 10^{-5} & 0 & 7.780 \times 10^{-5} \\
            3.964 \times 10^{-6} &  0  &   1.321 \times 10^{-5} & 0 \\
            0 & 1.211 \times 10^{-6} & 0 & 1.171 \times 10^{-5}
        \end{bmatrix},\\ %, \quad Q = 5000 I_6, \quad R = 2 \times 10^{-4} I_4
        Q &= 5000 I_6, \quad R = 2 \times 10^{-4} I_4, \quad L = 0,\\
        A_0^i &= 0.012 \times i \times  \frac{\| A \|_{\infty}}{\| \widehat{A}_0^i \|_{\infty}} \widehat{A}_0^i \ \text{ with }\ \widehat{A}_0^i = \mbox{\texttt{wgn}}(6,6,100 \times i), \mbox{ for } i = 1, 2, 3, \\
        B_0^i &= 0.012 \times i \times  \frac{\| B \|_{\infty}}{\| \widehat{B}_0^i \|_{\infty}} \widehat{B}_0^i \ \text{ with }\ \widehat{B}_0^i = \mbox{\texttt{wgn}}(6,4,40 \times i), \mbox{ for } i = 1, 2, 3.
    \end{align*} 
 %   at some state $\mathbf{x}$. Set
 %   \begin{align*}
%        \widehat{A}_1 &= \mbox{\texttt{wgn}}(6,6,100), \quad A_0^1 = 0.15 \times  \frac{\| A \|_{\infty}}{\| \widehat{A}_1 \|_{\infty}} \widehat{A}_1,  \quad
%      \widehat{B}_1 = \mbox{\texttt{wgn}}(6,4,40), \quad B_0^1 = 0.25 \times  \frac{\| B \|_{\infty}}{\| \widehat{B}_1 \|_{\infty}} \widehat{B}_1.
%    \end{align*}
    %where $A$, $B$, $Q$, $R$ and $L$ are in \cite{cpws:2022} for controlling of F16 multirole aircraft.
\end{example}

\begin{example}{\cite{chhu:2022}} \label{example:Quadrotor}
    The SDRE optimal control design for quadrotors for enhancing robustness against unmodeled disturbances \cite{chhu:2022} can be described as in \eqref{eq:dyna_SSDC} and \eqref{eq:cost_fun} with
    % \begin{align*}
    %     \dot{\mathbf{x}} =& A(\mathbf{x}) \mathbf{x} + B  \mathbf{u}, \\
    %     \min & \frac{1}{2} \int_0^{\infty} \left[ \mathbf{e}^{\top} Q \mathbf{e} + \mathbf{u}^{\top} R \mathbf{u} \right] dt
    % \end{align*}
    $\mathbf{x} = [u, v, w, p, q, r, \phi, \theta, z_a]^{\top}$, $\mathbf{e} = \mathbf{x} - \mathbf{x}_f$, 
    \begin{align*}
        A(\mathbf{x}) &= \begin{bmatrix}
            0 & \frac{r}{2} & -\frac{q}{2} & 0 & -\frac{w}{2} & \frac{v}{2} & 0 & -\frac{g(\sin \theta)}{\theta} & 0 \\
            -\frac{r}{2} & 0 & \frac{p}{2} & \frac{w}{2} & 0 & -\frac{u}{2} & \frac{g(1+\cos \theta) \sin \phi}{2\phi} & -\frac{g(1-\cos \theta) \sin \phi}{2\theta} & 0 \\
            \frac{q}{2} & -\frac{p}{2} & 0 & -\frac{v}{2} & \frac{u}{2} & 0 & -\frac{2g}{\phi}\sin^2(\frac{\phi}{2})\cos^2(\frac{\theta}{2}) & -\frac{2g}{\theta} \sin^2(\frac{\theta}{2}) \cos^2(\frac{\phi}{2}) & \frac{g}{z_a} \\
            0 & 0& 0 & 0 & \frac{c_1r}{2} & \frac{c_1q}{2} & 0 & 0 & 0 \\
            0 & 0 & 0 & \frac{c_2r}{2} & 0 & \frac{c_2p}{2} & 0 & 0 & 0 \\
            0 & 0 & 0 & \frac{c_3q}{2} & \frac{c_3p}{2} & 0 & 0 & 0 & 0 \\
            0 & 0 & 0 & 1 & \frac{\sin \phi \tan \theta}{3} & \frac{\tan \theta(1 + \cos \phi)}{3} & \frac{\alpha_1}{3\phi} & \frac{\alpha_2}{3\theta} & 0 \\
            0 & 0 & 0 & 0 & \frac{1+\cos \phi}{2} & -\frac{\sin \phi}{2} & -\frac{\alpha_3}{2\phi} & 0 & 0 \\
            0 & 0 & 0 & 0 & 0 & 0 & 0 & 0 & -\eta
        \end{bmatrix}, \\
        B  &= \begin{bmatrix}
            0 & 0 & -\frac{1}{m} & 0 & 0 & 0 & 0 & 0 & 0 \\
            0 & 0 & 0 & \frac{1}{I_x} & 0 & 0 & 0& 0 & 0 \\
            0 & 0 & 0 & 0 & \frac{1}{I_y} & 0 & 0 & 0 & 0 \\
            0 & 0 & 0 & 0 & 0 & \frac{1}{I_z} & 0 & 0 & 0 
        \end{bmatrix}^{\top},\\ 
        Q &= \mbox{\rm diag}(2000, 2000, 3000, 10, 10, 100, 0, 0, 0), \ R = I_4, \ L = 0, \\
        A_0^i &= 0.025 \times i \times  \frac{\| A \|_{\infty}}{\| \widehat{A}_0^i \|_{\infty}} \widehat{A}_0^i\ \text{ with }\ \widehat{A}_0^i = \mbox{\texttt{wgn}}(9,9,10\times i), \mbox{ for } i = 1, 2, 3, \\ 
        B_0^i &= 0.01 \times i \times  \frac{\| B \|_{\infty}}{\| \widehat{B}_0^i \|_{\infty}} \widehat{B}_0^i \ \text{ with } \ 
      \widehat{B}_0^i = \mbox{\texttt{wgn}}(9,4,4\times i), \mbox{ for } i = 1, 2, 3, 
    \end{align*}
    where $\alpha_1 = q\tan \theta \sin \phi - r \tan \theta + r \tan \theta \cos \phi$, $\alpha_2 = q \tan \theta \sin \phi + r \tan \theta \cos \phi$, and $\alpha_3 = q(1-\cos \phi) + r \sin \phi$, 
    $g$ is the gravity force, $m$ is the quadrotor
mass, $(u,v,w)$ and $(p,q,r)$ are the velocity and the angular velocity on the body-fixed frame, respectively, for roll $\phi$, pitch $\theta$ and yaw $\psi$ angles,
    $z_a$ is a slow varying stable auxiliary variable governed by $\dot{z}_a = -\eta z_a$, $\eta > 0$, $c_1 = \frac{I_y - I_z}{I_x}$, $c_2 = \frac{I_z - I_x}{I_y}$, $c_3 =  \frac{I_x - I_y}{I_z}$, and $I_x$, $I_y$, $I_z$,  are initial parameters.
    
    For a fixed state $\mathbf{x}$ with $m = 1$,  $I_x = I_y = 0.01466$, and $I_z = 0.02848$, the coefficient matrix $A \equiv A(\mathbf{x})$ has the form
    \begin{small}
    \begin{align*}
        A &= \begin{bmatrix}
            0 & -8.208\mbox{e-}4  & -1.047\mbox{e-}2     &       0 & -1.234\mbox{e-}4 &  1.178  &          0 & -9.8000     &       0 \\
   8.208\mbox{e-}4     &       0  & -1.603\mbox{e-}3 &  1.234\mbox{e-}4    &        0  & 2.203\mbox{e-}2  & 9.800 & -5.436\mbox{e-}4     &       0 \\
   1.047\mbox{e-}2 &  1.603\mbox{e-}3    &        0 & -1.178 & -2.203\mbox{e-}2    &        0    &        0     &       0  &  9.820\mbox{e-}1 \\
            0     &       0     &       0     &       0  & 7.738\mbox{e-}4 &  -9.871\mbox{e-}3  &          0      &      0      &      0 \\
            0      &      0     &       0  & -7.738\mbox{e-}4    &        0  & -1.511\mbox{e-}3 &           0     &       0      &      0 \\
            0      &      0     &       0      &      0     &       0     &       0    &        0     &       0     &       0 \\
            0     &       0     &       0  & 1.000 &  1.386\mbox{e-}8 &  2.499\mbox{e-}4  & 2.617\mbox{e-}6 & -5.464\mbox{e-}4    &        0 \\
            0     &       0     &       0      &      0  & 1.000    &        0  & -9.650\mbox{e-}3      &      0      &      0 \\
            0    &        0    &        0     &       0       &     0      &      0    &        0       &     0  & -0.100 
        \end{bmatrix}.
    \end{align*}
    \end{small}

    % \begin{align*}
    %     A &= \begin{bmatrix}
    %         0 & 0 & 0 & 0 & 0 & 1.178 & 0 & -9.800 & 0 \\
    %         0 & 0 & 0 & 0 & 0 & 6.940 \times 10^{-3} & 9.800 & -8.552\times 10^{-4} & 0 \\
    %         0  & 0  & 0 & -1.178 &  -6.940 \times 10^{-3} & 0 & 0  & 0 & 0.9800 \\
    %         0 & 0 & 0 & 0 & 0 & 0 & 0 & 0 & 0 \\
    %         0 & 0 & 0 & 0 & 0 & 0 & 0 & 0 & 0 \\
    %         0 & 0 & 0 & 0 & 0 & 0 & 0 & 0 & 0 \\
    %         0 & 0 & 0 & 1.000 & 1.015\times 10^{-8}  & 1.164 \times 10^{-4} & 0 & 0 & 0 \\
    %         0 & 0 & 0 & 0 & 1.000 & 0 & 0 & 0 & 0 \\
    %         0 & 0 & 0 & 0 & 0 & 0 & 0 & 0 & -0.100
    %     \end{bmatrix}, \\
    %     B &= \begin{bmatrix}
    %         0 & 0 & 0 & 0\\
    %         0 & 0 & 0 & 0 \\
    %         -1.000 & 0 & 0 & 0 \\
    %         0 & 68.21 & 0 & 0 \\
    %         0 & 0 & 68.21 & 0 \\
    %         0 & 0 & 0 & 35.11 \\
    %         0 & 0 & 0 & 0\\
    %         0 & 0 & 0 & 0\\
    %         0 & 0 & 0 & 0
    %     \end{bmatrix}, \quad Q = \mbox{\rm diag}(2000, 2000, 3000, 10, 10, 100, 0, 0, 0), \ R = I_4, \ L = 0, \\
    %      A_0^1 &= 0.25 \times  \frac{\| A \|_{\infty}}{\| \widehat{A}_0^1 \|_{\infty}} \widehat{A}_0^1\text{ with }\widehat{A}_0^1 = \mbox{\texttt{wgn}}(9,9,10),  \quad B_0^1 = 0.25 \times  \frac{\| B \|_{\infty}}{\| \widehat{B}_0^1 \|_{\infty}} \widehat{B}_0^1\text{ with }
    %   \widehat{B}_0^1 = \mbox{\texttt{wgn}}(9,4,4), 
    % \end{align*}
    %where $A$, $B$, $Q$, $R$ and $L$ are in \cite{chhu:2022} for state-feedback control of quadrotor.
\end{example}
%\begin{example} \label{example:5}
%Let
%\begin{align*}
%    A &= \varepsilon \begin{bmatrix}
%         \frac{7}{3} & \frac{2}{3} & 0 \\
%         \frac{2}{3} & 2 & - \frac{2}{3} \\
%         0 & -\frac{2}{3} & \frac{5}{3}
%    \end{bmatrix}, \quad B = \frac{1}{\sqrt{\varepsilon}} I_3, \quad A_1 =   \varepsilon  \begin{bmatrix}
%         0.5 & -0.5 & 0.05 \\  -1 & 0.5 & -0.5 \\  0.25 & -0.05& 1.5
%    \end{bmatrix},  \quad R = I_3, \quad L = 0,\\ 
%    Q &= \begin{bmatrix}
%         (4\varepsilon+4+\varepsilon^{-1})/9 &  2(2\varepsilon-1-\varepsilon^{-1})/9 & 2(2-\varepsilon-\varepsilon^{-1})/9 \\
%         2(2\varepsilon-1-\varepsilon^{-1})/9 & (1+4\varepsilon+4/\varepsilon)/9  &    2(-1-\varepsilon+2/\varepsilon)/9 \\
%                      2(2-\varepsilon-\varepsilon^{-1})/9 &  2(-1-\varepsilon+2/\varepsilon)/9  &   (4+\varepsilon+4/\varepsilon)/9
%    \end{bmatrix}, \quad B_1 = \frac{1}{\sqrt{\varepsilon}} \begin{bmatrix}
%         0 & 0 & 0.2 \\ 0.36 &  -0.6 & 0 \\  0 & -0.95 & -0.032
%    \end{bmatrix}.
%\end{align*}
%\end{example}

In Examples~\ref{example:6} -  \ref{example:Quadrotor}, we use $X_0 = 0$ as the initial matrix so that $\{ X_k \}_{k=0}^{\infty}$ is monotonically non-decreasing for the FP-CARE\_SDA algorithm.
%randomly construct an initial $X_0$ such that $\| X_0 \|_F$ is sufficiently larger than $\| \widehat{X} \|_F$.  
The associated convergence of $\mbox{NRes}_k$ for these examples is presented in Figure~\ref{fig:1_k2_convergence}.  Numerical results show that
\begin{itemize}
    %
    %\item For the fixed-point iteration, all four examples are linearly convergent. But, the iteration numbers are obviously larger than that of the other methods.
    %
    %\item For the FP-CARE\_SDA algorithm, $X_0$ does not satisfy the conditions of Theorem~\ref{thm:monotonical_decreasing} so that the nonincreasing monotonic property is not preserved. However, after a few $m$ iterations, $\{ X_k \}_{k = m+1}^{\infty}$ is nonincreasing or nondecreasing monotonic.
    %
    \item It is well known that Newton's method (NT-FP-Lyap\_SDA algorithm) is highly dependent on the initial choice $X_0$. 
    When $X_0 = 0$ or $X_0\geq 0$ is randomly constructed, 
    %is randomly constructed so that 
    NT-FP-Lyap\_SDA does not converge to $\widehat{X} \geq 0$ for these four examples. 
    However, in Figure~\ref{fig:1_k2_convergence}, after a few $m$ iterations of FP-CARE\_SDA for obtaining $X_m$ with $\| X_m - X_{m-1} \|_2 \leq 0.01$,  $X_m$ is closed to $\widehat{X}$ and Newton’s method with initial $X_m$ (i.e. FPC-NT-FP-Lyap\_SDA algorithm) has quadratic convergence. 
    \item Numerical results in Figure~\ref{fig:1_k2_convergence} show that the fixed-point method, FP-CARE\_SDA and FPC-mNT-FP-Lyap\_SDA algorithms have linear convergence. But, the iteration number of the fixed-point method is obviously larger than that of the other methods.
    Comparing the results in Figure~\ref{fig:1_k2_convergence} with those in Figure~\ref{fig:spectial_X0_convergence}, the iteration number of mNT-FP-Lyap\_SDA algorithm in Figure~\ref{fig:spectial_X0_convergence} is obviously larger than that of FP-CARE\_SDA algorithm. 
    As shown in \eqref{eq:equiv_mNT}, the quadratic term $(X-X_k)G_k(X-X_k)$ is small when $X_m$ is used as an initial matrix of mNT-FP-Lyap\_SDA algorithm.
    We can expect that the mNT-FP-Lyap\_SDA algorithm with initial $X_m$ (i.e. FPC-mNT-FP-Lyap\_SDA algorithm) has the same convergent behavior of the FP-CARE\_SDA algorithm as presented in Figure~\ref{fig:1_k2_convergence}. 
    %$X_m$ is a good approximated solution of $\widehat{X}$. 
    Without such initial $X_m$ (i.e., with $X_0 = 0$ or a randomly initial $X_0$), mNT-FP-Lyap\_SDA algorithm will not converge for Examples~\ref{example:6} -  \ref{example:Quadrotor}. This demonstrates the importance and robustness of the FP-CARE\_SDA algorithm.
    %
    %\item In Example \ref{example:6}, even $X_0$ does not satisfy the conditions in Theorem~\ref{thm:monotonical_decreasing}, numerical results show that $X_0 \geq X_1$ for FP-CARE\_SDA algorithm.  By mathematical induction in the proof of FP-CARE\_SDA algorithm, we have $X_0 \geq X_1 \geq X_2 \geq \cdots$. Moreover, the numerical result also shows that $X_m$ not only satisfies the conditions in Theorem~\ref{thm:monotonical_decreasing} but also satisfies the conditions in Theorem~\ref{thm:convergence_Newton}, which leads to the non-increasing monotonic convergence of FP-CARE\_SDA, FPC-NT-FP-Lyap\_SDA and FPC-mNT-FP-Lyap\_SDA algorithms.
    %
    \item For Examples~\ref{example:3} -  \ref{example:Quadrotor}, %after a few iterations of the FP-CARE\_SDA algorithm, the FP-CARE\_SDA and 
    FPC-mNT-FP-Lyap\_SDA algorithm produces monotonically non-decreasing sequences and has a linear convergence. The sequence produced by FPC-NT-FP-Lyap\_SDA algorithm does not contain any monotonic property. 
\end{itemize}

From the results in Figures~\ref{fig:spectial_X0_convergence} and \ref{fig:1_k2_convergence}, 
%no matter $X_0$ satisfying the conditions of Theorem~\ref{thm:monotonical_decreasing} or not, 
the FP-CARE\_SDA algorithm has a linear convergence. Moreover, the FP-CARE\_SDA algorithm also provides a good initial matrix for the NT-FP-Lyap\_SDA algorithm and the modified Newton's iteration in \eqref{eq:m_Newton_iter} so that the FPC-NT-FP-Lyap\_SDA and FPC-mNT-FP-Lyap\_SDA algorithms have quadratic and linear convergence, respectively. 
These two algorithms can not compute the solution $\widehat{X}$ without such an initial matrix. Therefore, our proposed FP-CARE\_SDA algorithm is a reliable and robust algorithm.

\subsection{Efficiency of algorithms} \label{subsec:efficiency}
In each iteration of the FP-CARE\_SDA (Algorithm~\ref{alg:SDA-CARE}), FPC-NT-FP-Lyap\_SDA (Algorithm~\ref{alg:CARE Newton fixed-point Lyapunov}), and FPC-mNT-FP-Lyap\_SDA (Algorithm~\ref{alg:CARE mNewton Lyapunov}), solving the CARE and Lyapunov equation are the most computationally intensive steps.  To compare the efficiency of these methods, we first show the total numbers of solving CARE \eqref{eq:kth-SARE_CARE}, Lyapunov equation \eqref{eq:Lyapunov}/\eqref{eq:m_Newton_iter}, and fixed-point iteration \eqref{eq:FP-SCARE} for each method in Table~\ref{tab:no_CARE_Lyap}. The third and fifth columns show that we only need a few iterations of the FP-CARE\_SDA algorithm to get a good initial matrix in FPC-NT-FP-Lyap\_SDA and FPC-mNT-FP-Lyap\_SDA algorithms. 

\begin{table}[h]
  \centering
  \begin{tabular}{|c|c|c|c|c|c|c|c|c|} \hline
%    & \multicolumn{4}{ c| }{CPU times (seconds)} \\ \cline{2-5} %& \multicolumn{4}{ c| }{Iteration numbers} \\ \cline{2-9}
    & FP-CARE\_SDA & \multicolumn{2}{ c| }{FPC-NT-FP-Lyap\_SDA} & \multicolumn{2}{ c| }{FPC-mNT-FP-Lyap\_SDA} & FP \\ \hline %& FP-CARE & Newton-FP-Lyap & FPC-Newton-FP-Lyap & FP \\ \hline
    & solving \eqref{eq:kth-SARE_CARE} & solving \eqref{eq:kth-SARE_CARE} & solving \eqref{eq:Lyapunov} & solving \eqref{eq:kth-SARE_CARE} & solving \eqref{eq:m_Newton_iter} & solving \eqref{eq:FP-SCARE} \\ \hline 
%    Exa. \ref{example:1} & 19 & 5 & 17 & 5 & 14 & 34\\
    %
%    Exa. \ref{example:12} & 7 & 2 & 8 & 2 & 5 & 764 \\
    %
    Exa. \ref{example:6} & 14 & 3 & 16 & 3 & 11 & 86 \\
    %
%    Exa. \ref{example:2} & 2.441e-3 & $\times$ & 2.925e-3 & 2.258e-3\\
    Exa. \ref{example:3} & 83 & 10 & 111 & 10 & 74 & 185 \\
    %
%    Exa. \ref{example:7} & 50 & 5 & 56 & 5 & 45 & 242 \\
    %
    Exa. \ref{example:F16} &74 & 8 & 114 & 8 & 71 & 410 \\
    Exa. \ref{example:Quadrotor} & 85 & 10 & 107 & 10 & 74 & 300 \\ \hline
  \end{tabular}
  \caption{Number of solving CARE \eqref{eq:kth-SARE_CARE}, Lyapunov equation \eqref{eq:Lyapunov}/\eqref{eq:m_Newton_iter}, and fixed-point (FP) iteration \eqref{eq:FP-SCARE}.}
  \label{tab:no_CARE_Lyap}
\end{table}

\begin{table}[h]
  \centering
  \begin{tabular}{|c|c|c|c|c|c|c|c|c|} \hline
    & \multicolumn{4}{ c| }{CPU times (seconds)} \\ \cline{2-5} %& \multicolumn{4}{ c| }{Iteration numbers} \\ \cline{2-9}
    & FP-CARE\_SDA & FPC-NT-FP-Lyap\_SDA & FPC-mNT-FP-Lyap\_SDA & FP \\ \hline %& FP-CARE & Newton-FP-Lyap & FPC-Newton-FP-Lyap & FP \\ \hline
%    Exa. \ref{example:1} & 9.150e-4 & 6.871e-4 & 6.255e-4 & 7.482e-4\\
    %
%    Exa. \ref{example:12} & 4.214e-4 & 3.666e-4 & 2.803e-4 & 1.327e-2 \\
    %
    Exa. \ref{example:6} & 7.947e-1 & 8.112e-1& 5.729e-1  & 5.790  \\
    Exa. \ref{example:3} & 6.151e-3 & 4.632e-3 & 3.928e-3  & 9.634e-3\\
    %
%    Exa. \ref{example:7} & 2.813 & 1.884& 1.696  & 2.885 \\
    %
    Exa. \ref{example:F16} & 5.527e-3 & 4.897e-3 & 3.816e-3  & 2.041e-2 \\
    Exa. \ref{example:Quadrotor} & 1.217e-2 & 7.759e-3 & 6.627e-3  & 3.804e-2 \\ \hline
  \end{tabular}
  \caption{CPU times for solving SCARE \eqref{eq:Stoc-ARE-N}.}
  \label{tab:CPU_time}
\end{table}

\begin{itemize}
\item The fourth and sixth columns of Table \ref{tab:no_CARE_Lyap} show that the FPC-NT-FP-Lyap\_SDA algorithm still has a quadratic convergence, the convergence of the fixed-point iteration \eqref{eq:Lyapunov} for  the FPC-NT-FP-Lyap\_SDA algorithm is linear so that the total number of solving Lyapunov equations is larger than that for the FPC-mNT-FP-Lyap\_SDA algorithm. As shown in Table~\ref{tab:CPU_time}, the performance of the FPC-mNT-FP-Lyap\_SDA algorithm is better than that of the FPC-NT-FP-Lyap\_SDA algorithm. 
\item Sum of the numbers in the fifth and sixth columns of Table~\ref{tab:no_CARE_Lyap} is almost equal to the number in the second column. However, the computational cost for solving CARE \eqref{eq:kth-SARE_CARE} by SDA in \cite{hull:2018} is larger than for solving Lyapunov equation~\eqref{eq:m_Newton_iter} by L-SDA in Algorithm~\ref{alg:Lyapunov_SDA}. This indicates that the FPC-mNT-FP-Lyap\_SDA algorithm outperforms the FP-CARE\_SDA algorithm as shown in the comparison of the CPU time in Table~\ref{tab:CPU_time}. 
\item The dimension of the coefficient matrix for the linear system of the fixed-point iteration in \eqref{eq:FP-SCARE} is $n(r+1)$. Compared to the other three algorithms, in which matrix dimensions are $n$, the fixed-point iteration needs more computational cost to solve the linear systems. On the other hand, as shown in Figures~\ref{fig:spectial_X0_convergence}-\ref{fig:1_k2_convergence} and the last column in Table~\ref{tab:no_CARE_Lyap}, the fixed-point iteration needs a large iteration number to get the solution. These lead to the performance of the fixed-point iteration being worse than that of the other three algorithms as shown in Table~\ref{tab:CPU_time}. 
\end{itemize}

Summarized above results, we can see that the FPC-mNT-FP-Lyap\_SDA algorithm outperforms the FP-CARE\_SDA, FPC-NT-FP-Lyap\_SDA algorithms, and fixed-point iteration. 

\section{Conclusions}
We propose a reliable FP-CARE\_SDA for solving the SCARE \eqref{eq:SCARE_gen} and prove  its monotonically non-decreasing convergence with $X_0 = 0$, its monotonically non-increasing convergence with $\mathcal{R}(X_0) \leq 0$ and $A_0 -G_0X_0$ being stable.  
Furthermore, we propose the mNT-FP-Lyap\_SDA algorithm to be used to accelerate the convergence with the FP-CARE\_SDA as a robust initial step and prove its monotonically non-decreasing convergence under the assumption that the resulting $X_k \geq 0$.
Numerical experiments of real-world practical applications from the 3D missile/target engagement, the F16 aircraft flight control, and the quadrotor optimal control show the robustness of our proposed FPC-mNT-FP-Lyap\_SDA algorithm.

\section{Appendix}
In \cite{guli:2023}, the authors proposed a fixed-point iteration to solve stochastic discrete-time algebraic Riccati equations (SDAREs) and proved the convergence of the fixed-point iteration. Moreover, the M\"{o}bius transformation is applied to transform the SCARE \eqref{eq:Stoc-ARE-N} into a SDARE so that the proposed fixed-point iteration can be applied to solve the associated solution $X$. Now, we state the fixed-point iteration for solving SCARE. Let 
\begin{align*}
      \widehat{A} = A - B R^{-1} L^{\top}, \quad \widehat{B} = B R^{-1/2}, \quad \widehat{C}^{\top} \widehat{C} = Q - L R^{-1} L^{\top},
\end{align*}
$\Pi$ and $\widehat{\Pi} \in \mathbb{R}^{n(r+1) \times n(r+1)}$ be the permutations satisfying 
\begin{align*}
     \Pi^{\top}( X \otimes I_r) \Pi &= I_r \otimes X, \\
     \mbox{diag}\left( X, X \otimes I_r \right) &= \widehat{\Pi}^{\top} \left( X \otimes I_{r+1} \right) \widehat{\Pi}, 
\end{align*}
and
\begin{align*}
    \mathcal{A} = \Pi \left( \begin{bmatrix}
          A_0^1 \\ \vdots \\ A_0^r 
    \end{bmatrix} - \begin{bmatrix}
          B_0^1 \\ \vdots \\ B_0^r
    \end{bmatrix} R^{-1} L^{\top}\right), \quad \mathcal{B} = \Pi \begin{bmatrix}
          B_0^1 \\ \vdots \\ B_0^r
    \end{bmatrix} R^{-/12}.
\end{align*}
Give $\gamma > 0$ so that $\widehat{A}_{\gamma} \equiv \widehat{A} - \gamma I_n$ is nonsingular. Define $Z_{\gamma} = \widehat{C} \widehat{A}_{\gamma}^{-1} \widehat{B}$. By M\"{o}bius transformation, we have
\begin{subequations} \label{eq:mtx_FP}
\begin{align}
      E_{\gamma} &= \widehat{\Pi}\begin{bmatrix}
           \widehat{A}_{\gamma} + 2 \gamma I_n + \widehat{B} Z_{\gamma}^{\top} \widehat{C} \\
           \sqrt{2\gamma} (\mathcal{A} + \mathcal{B} Z_{\gamma}^{\top}\widehat{C})  
      \end{bmatrix} (I_n + \widehat{A}_{\gamma}^{-1} \widehat{B}Z_{\gamma}^{\top} \widehat{C})^{-1} \widehat{A}_{\gamma}^{-1} \in \mathbb{R}^{n(r+1) \times n}, \\
      H_{\gamma} &= 2 \gamma \widehat{A}_{\gamma}^{-\top} \widehat{C}^{\top} ( I_{\ell} + Z_{\gamma} Z_{\gamma}^{\top})^{-1} \widehat{C} \widehat{A}_{\gamma}^{-1} \in \mathbb{R}^{n \times n}, \\
      G_{\gamma} &= \widehat{\Pi} \begin{bmatrix}
            \sqrt{2\gamma} \widehat{A}_{\gamma}^{-1} \widehat{B} \\
            \mathcal{A} \widehat{A}_{\gamma}^{-1}\widehat{B} - \mathcal{B}
      \end{bmatrix} ( I_m + Z_{\gamma}^{\top} Z_{\gamma})^{-1} \begin{bmatrix}
            \sqrt{2\gamma} \widehat{A}_{\gamma}^{-1} \widehat{B} \\
            \mathcal{A} \widehat{A}_{\gamma}^{-1} \widehat{B} - \mathcal{B}
      \end{bmatrix}^{\top} \widehat{\Pi}^{\top} \in \mathbb{R}^{n(r+1) \times n(r+1)}.
\end{align}
\end{subequations}
Then, the fixed-point iteration in \cite{guli:2023} is given as
\begin{align}
     X_1 & = H_{\gamma}, \nonumber \\
     X_{k+1} &= E_{\gamma}^{\top}(X_k \otimes I_{r+1}) \left(I_{n(r+1)} + G_{\gamma}(X_k \otimes I_{r+1})\right)^{-1} E_{\gamma} + H_{\gamma} \label{eq:FP-SCARE}
\end{align}
for $k = 1, 2, \ldots$.

Each iteration in \eqref{eq:FP-SCARE} only needs to compute $(X_k \otimes I_{r+1}) \left(I_{n(r+1)} + G_{\gamma}(X_k \otimes I_{r+1})\right)^{-1} E_{\gamma}$  by the direct method. However, the dimension of the coefficient matrix is enlarged to $n(r+1) \times n(r+1)$. It needs to have more and more computational cost as $r$ is large. On the other hand, the iteration can not preserve the symmetric property of the solution.

\section*{Acknowledgments}
This work was partially supported by the National Science and Technology Council (NSTC), the National Center for Theoretical Sciences, and the Nanjing Center for Applied Mathematics.  T.-M. Huang, Y.-C. Kuo, and W.-W. Lin were partially supported by
NSTC 110-2115-M-003-012-MY3, 110-2115-M-390-002-MY3 and 112-2115-M-A49-010-, respectively.

%{\small
%\bibliographystyle{plain}
%\bibliography{\TeXHOME/BIB/strings,\TeXHOME/BIB/mxptrefs,\TeXHOME/BIB/li,\TeXHOME/BIB/ML}
%}

\bibliographystyle{abbrv}
\bibliography{strings,research_papers}
%\bibliography{/Users/min/Dropbox/ActaNUmerica/tex/book/strings,/Users/min/Dropbox/ActaNUmerica/tex/book/research_papers}

\end{document}